\documentclass{article}

\usepackage{amsmath,amsthm,amsfonts,amssymb}
\usepackage{mathtools}
\newtheorem{theorem}{Theorem}[section]
\newtheorem{lemma}[theorem]{Lemma}
\newtheorem{corollary}[theorem]{Corollary}
\newtheorem{proposition}[theorem]{Proposition}
\theoremstyle{definition}
\newtheorem{remark}[theorem]{Remark}
\newtheorem{assumption}[theorem]{Assumption}

\usepackage{enumerate}
\usepackage{algorithm}
\usepackage[noend]{algpseudocode}
\usepackage{hyperref}
\usepackage{xcolor}
\usepackage{bbm}
\usepackage{pgfplots}
\usepackage{tikz}
\usepackage{tikz-3dplot}
\usetikzlibrary{3d,calc}

\usepackage{booktabs}
\usepackage{multirow}

\usepackage{listings}

\DeclareFixedFont{\ttb}{T1}{txtt}{bx}{n}{12} 
\DeclareFixedFont{\ttm}{T1}{txtt}{m}{n}{12}  

\definecolor{codegreen}{rgb}{0,0.6,0}
\definecolor{codegray}{rgb}{0.5,0.5,0.5}
\definecolor{codepurple}{rgb}{0.58,0,0.82}
\definecolor{backcolour}{rgb}{0.95,0.95,0.92}

\lstdefinestyle{mystyle}{
    commentstyle=\color{codegreen},
    keywordstyle=\color{magenta},
    numberstyle=\tiny\color{codegray},
    stringstyle=\color{codepurple},
    basicstyle=\ttfamily\footnotesize,
    breakatwhitespace=false,         
    breaklines=true,                 
    captionpos=b,                    
    keepspaces=true,                 
    showspaces=false,                
    showstringspaces=false,
    showtabs=false,                  
    tabsize=2
}
\newcommand{\EE}{\mathbb{E}}
\newcommand{\PP}{\mathbb{P}}
\newcommand{\RR}{\mathbb{R}}
\newcommand{\NN}{\mathbb{N}}
\newcommand{\sphere}{\mathbb{S}}
\newcommand{\dd}{\mathrm{d}}

\newcommand{\cN}{\mathcal{N}}

\newcommand{\cS}{\mathcal{S}}
\newcommand{\cF}{\mathcal{F}}
\newcommand{\bigO}{\mathcal{O}}

\newcommand{\scp}[2]{\left\langle{#1}, {#2}\right\rangle}
\DeclarePairedDelimiter{\abs}{\vert}{\vert}
\DeclarePairedDelimiter{\norm}{\|}{\|}

\DeclareMathOperator{\sign}{sign}
\DeclareMathOperator{\diag}{diag}
\DeclareMathOperator{\Spann}{span}
\let\dim\relax
\DeclareMathOperator{\dim}{dim}
\DeclareMathOperator*{\argmax}{argmax}
\DeclareMathOperator*{\argmin}{argmin}
\DeclareMathOperator*{\dist}{dist}

\title{Matrix-free stochastic calculation of operator norms without using adjoints}
\author{Jonas Bresch\thanks{Institute of Mathematics, Technical University Berlin, Straße des 17. Juni 136, 10623 Berlin, Germany,
\href{mailto:bresch@math.tu-berlin.de}{bresch@math.tu-berlin.de}. \url{https://www.tu.berlin/imageanalysis}.} \and Dirk A. Lorenz\thanks{Center for Industrial Mathematics, Fachbereich 3, University of Bremen, Postfach 33 04 40, 28334 Bremen, Germany, \href{mailto:d.lorenz@uni-bremen.de}{d.lorenz@uni-bremen.de}.} \and Felix Schneppe\thanks{Center for Industrial Mathematics, Fachbereich 3, University of Bremen, Postfach 33 04 40, 28334 Bremen, Germany, \href{mailto:schneppe@uni-bremen.de}{schneppe@uni-bremen.de}.} \and Maximilian Winkler\thanks{Center for Industrial Mathematics, Fachbereich 3, University of Bremen, Postfach 33 04 40, 28334 Bremen, Germany, \href{mailto:maxwin@uni-bremen.de}{maxwin@uni-bremen.de}.}}

\begin{document}
\maketitle

\begin{abstract}
  This paper considers the problem of computing the operator norm 
  of a linear map between finite dimensional Hilbert spaces 
  when only evaluations of the linear map are available 
  and under restrictive storage assumptions.
  We propose a stochastic method of random search type 
  to maximize the Rayleigh quotient 
  and employ an exact line search in the random search directions.
  Moreover, we show that the proposed algorithm converges 
  to the global maximum (the operator norm) almost surely and a sublinear convergence behavior for the corresponding 
  eigenvector and eigenvalue equation.
  Finally, we illustrate the performance of the method with numerical experiments.
\end{abstract}



\section{Introduction}
\label{sec:intro}

In this paper, we revisit the problem of the computation of the operator norm $\norm{A}$ of a linear map $A$ between Hilbert spaces of dimension $d$ and $m$, respectively.
We reconsider this problem under restrictive conditions for which we could not find a satisfactory solution in the literature and which wwe encountered in our analysis of primal-dual algorithms with mismatched adjoint in the context of computerized tomography.
First, we assume that the map $A$ is only available through a black-box oracle that returns $Av$ for a given input $v$,
e.g. there is precompiled code available which evaluates a linear map.
Moreover, we consider that it is not possible to store a large number of vectors (of either the input or the output dimension),
which means that we assume that the method has a storage requirement of $\bigO(\max(m,d))$.
Finally, we are interested in a method that is able to approximate the value $\norm{A}$ to any accuracy and also is able to make further incremental improvements if the output at some stage is not good enough.
We describe, why we are interested in this scenario in the next section.

\subsection{Motivation}
\label{sec:motivation}

Operator norms are needed in many places, e.g. to guarantee convergence of optimization methods:
The Lipschitz constant of the objective $\norm{Ax-y}^{2}$ is $\norm{A}^{2}$ and gradient descent or the forward-backward method need a stepsize $0<\tau<2/\norm{A}^{2}$ for convergence \cite{ryu2022large}.
The Chambolle-Pock method \cite{chambolle2011first} to minimize $f(x) + g(Ax)$ for convex, closed and proper functions $f$ and $g$ and a linear map $A$ needs two stepsizes $\sigma,\tau$ that fulfill $\sigma\tau\leq 1/\norm{A}^{2}$ \cite{bredies2022degenerate}).
When the operator $A$ is not given as a matrix, one usually resorts to the power iteration to compute $\norm{A}$ and when $A$ is not self-adjoint, one computes $\sqrt{\norm{A^{*}A}}$ by power iteration.
However, this poses an additional challenge: One needs the adjoint of $A$ as well.
If $A$ comes from the discretization of some operator, the discretization of the adjoint operator is usually not the adjoint of the discretized operator. This is discussed for numerical methods for partial differential equations,
e.g., in~\cite{fikl2016comprehensive} or \cite[Section 2.9]{gunzburger2002perspectives} and it is also a common issue in computed tomography where forward-, back-projector pairs are frequently not adjoint, see e.g.~\cite{bredies2021convergence,huber2025convergence}.\footnote{If the adjoint is not available or inexact, the optimization methods cannot be applied or are not guaranteed to converge. This is another problem and discussed, e.g. in~\cite{lorenz2023chambolle,lorenz2023randomdescent,chouzenoux2023convergence,Chouzenout2021pgm-adjoint}.}

  \paragraph{Leading example: Computed Tomography.}
  A concrete example where our restrictive assumptions hold is computed tomography, where the use of ``mismatched'' adjoints is common. In this application, the forward operator $A$ and the back-projector $V^*$ are often implemented as separate, highly optimized software routines (or ``black-box'' oracles) where $V^* \neq A^*$. This discrepancy stems from different discretizations schemes used to optimize memory and computational speed on GPUs~\cite{van2016fast, Hansen2021GMRESMF,huber2025convergence}. The operators $A^*$ and $V$ are therefore often only available through a black box oracle.
  Moreover, the input and output dimensions can be in the number of millions making it impossible to store a reasonable part of the system matrix.

While replacing the true adjoint $A^*$ with a mismatched operator $V^*$ allows for significant computational acceleration, it jeopardizes the convergence guarantees of standard variational methods. Recent theoretical advances have established that algorithms like the Chambolle-Pock method and the Primal-Dual Douglas-Rachford method can recover linear convergence even in the presence of adjoint mismatch, provided that step sizes are selected based on specific spectral properties of the operators, such as $\|V\|$ (the norm of the adjoint of the back-projector) and $\norm{A-V}$~\cite{lorenz2023chambolle, naldi2025pddrmm}. In the paper \cite{Chouzenout2021pgm-adjoint} the spectral properties of $L = V^*A+\kappa I$ are needed for the convergence of the proposed algorithm regarding the problem $\underset{x}{\operatorname{min}} \frac{1}{2}\|y-A x\|^2+g(x)+\frac{\kappa}{2}\|x\|^2$. In a mismatched setting where only black-box oracles for $A$ (forward) and $V^*$ (backward) are provided, the operator $V$ is implicit and unavailable. Consequently, the composition $V V^*$ required for the power iteration cannot be evaluated.

This application might look like a rarely occurring phenomenon, but the Radon transform implemented from the Python package \texttt{scikit-image}~\cite{scikit-image} provides a good example.
The function \texttt{skimage.transform.radon} implements a discretization of the linear operator given by the Radon transform.
The package does not provide the corresponding adjoint, but provides a \emph{filtered} back-projection in \texttt{skimage.transform.iradon} and when the filter is set to \texttt{None}, the \emph{unfiltered} backprojection is computed which is, in theory, the adjoint~\cite[Sections 6.5 \& 9.3.5]{Hansen2021CT}.
However, the implementation in \texttt{skimage} is far from being the exact adjoint
as can be checked easily by running the following code:\footnote{The masking is used since \texttt{radon} assumes that the input image is zeros outside the so-called reconstruction circle.}

\begin{lstlisting}[language=Python]
import numpy as np
from skimage.transform import radon, iradon

np.random.seed(4242)
N = 50
numAng = 70
theta = np.linspace(0.0, 180, numAng, endpoint=False)

v = np.random.randn(N,N)
mask = np.zeros((N,N), dtype=bool)
X,Y = np.meshgrid(np.arange(N), np.arange(N))
mask[(X-N/2)**2 + (Y-N/2)**2 <= (N/2)**2] = True
v[~mask] = 0
w = np.random.randn(N,numAng)
Av = radon(v,theta=theta)

ATw = iradon(w, theta=theta, filter_name=None)


print(f'<v,A^Tw> = {np.sum(v*ATw)},  <Av,w> = {np.sum(Av*w)}')
\end{lstlisting}
which gives the output
\begin{lstlisting}[language=Python]
<v,A^Tw> = -17.458356056194354,  <Av,w> = -754.3125085578928
\end{lstlisting}

We do not have a full explanation for the mismatch, but simple scaling can not explain the difference, since the quotient of \texttt{<v,ATw>} and \texttt{<Av,w>} changes for different realizations.

Since the norm of the back-projection is critical in establishing appropriate step-sizes in \cite{lorenz2023chambolle, naldi2025pddrmm}, algorithms, which calculate the operator norm without relying on the adjoint, are needed.\footnote{A method that can compute $\norm{A-V}$ when only black-box oracles for $A$ and $V^{*}$ are available is also needed, but this is treated in a different paper~\cite{bresch2025computing}.}

\paragraph{A curious toy example: Rotations.}
We provide a very simple illustration of the problem: Consider a function on the two dimensional disk $\mathbb{D} \coloneqq \{ x\in\RR^{2}\ \mid\ \norm{x}_{2}\leq 1 \}$ and the linear operator $A_{\alpha}$ that takes a function $f:\mathbb{D}\to \RR$ and maps it to the rotated function $A_{\alpha}f(x) \coloneqq f(R_{\alpha}x)$ where $R_{\alpha} \coloneqq \left(
\begin{smallmatrix}
  \cos(\alpha) & \sin(\alpha)\\-\sin(\alpha) & \cos(\alpha)
\end{smallmatrix}\right)
$. This operator is bounded and has norm $\norm{A_{\alpha}} = 1$ as an operator from $L^{2}(\mathbb{D})$ to itself. The adjoint operator is $A_{\alpha}^{*} = A_{-\alpha}$ and is also the inverse (the operator is orthogonal). When applied to functions that are given on an $n\times n$ grid, e.g. images with $d=n^{2}$ pixels, this rotation uses interpolation and we can neither expect that the dicretized inverse rotation is actually the inverse, nor that it is the adjoint.

This does affect the value of the operator norm: While the continuous operator fulfills $\norm{A_{\alpha}}=1$ for all $\alpha$, it is not clear if this is also fulfilled by the discretized map. Let us denote the discretized map with $\mathbf{A}_{\alpha}$. This is implemented in Python in \texttt{scipy.ndimage}, e.g. as
\begin{center}
  \texttt{rotate(x, angle, reshape=False, order=order)}
\end{center}
where \texttt{x} is the input array (and we assume that it is square and zero outside of a centered circle to avoid additional boundary effects) and \texttt{order} is the order of the interpolation routine (nearest neighbor, bilinear,\dots).
Now we compute the norm of $\mathbf{A}_{\alpha}$ by applying the power iteration to the (theoretically) self-adjoint operator $\mathbf{A}_{-\alpha}\mathbf{A}_{\alpha}$ and using the method in our paper. Without discretization error it would hold $\mathbf{A}_{-\alpha}\mathbf{A}_{\alpha} = I$, so the norm would be one.
Table~\ref{tab:table-norm-rotation} reports the values from computing $\sqrt{\norm{\mathbf{A}_{-\alpha}\mathbf{A}_{\alpha}}}$ with the power iteration and $\norm{\mathbf{A}_{\alpha}}$ with the method we propose in this paper for some values of $n$ and $\alpha$ and bicubic interpolation.
Note that the values larger than $1$ for the operator norm are not numerical errors: Our method constructs an array $\mathbf{v}$ with norm $1$ such that $\norm{\mathbf{A}_{\alpha}\mathbf{v}}$ has the values in the last column.
Notably, the relative error between $\sqrt{\norm{\mathbf{A}_{-\alpha}\mathbf{A}_{\alpha}}}$ and $\norm{\mathbf{A}_{\alpha}}$ can be as large as $17\%$ for $\alpha = 45^{\circ}$.
Even more drastically we get a norm of $1.4142$ for every angle different from $0^{\circ}$ and $90^{\circ}$ and any $n$ for the nearest neighbor interpolation.

\begin{table}[t]
\begin{minipage}{0.48\linewidth}
\begin{tabular}{rrrr}
\toprule
$n$ & $\alpha$ & $\sqrt{\norm{\mathbf{A}_{-\alpha}\mathbf{A}_{\alpha}}}$ & $\norm{\mathbf{A}_{\alpha}}$\\\midrule
25 & 10 & 0.99999 & 1.00463 \\
25 & 30 & 1.00002 & 1.05930 \\
25 & 45 & 0.99996 & 1.17387 \\
50 & 10 & 1.00080 & 0.99804 \\
50 & 30 & 0.99998 & 1.05314 \\
50 & 45 & 1.00000 & 1.16840 \\
\bottomrule
\end{tabular}
\end{minipage}
\hfill
\begin{minipage}{0.48\linewidth}
\caption{Numerical computation of operator norms for a rotation operation. The values in the column $\sqrt{\norm{\mathbf{A}_{-\alpha}\mathbf{A}_{\alpha}}}$ are obtained by power iteration while the values in the column $\norm{\mathbf{A}_{\alpha}}$ are computed with the method we propose in this paper.}
\label{tab:table-norm-rotation}
\end{minipage}
\end{table}

\paragraph{Another curious example: The Radon transform. }
Looking closer at the power iteration for $A^{*}A$ under the assumption that $V^{*}$ is used instead of the exact adjoint, one notes that there are different ways to approximate the norm of $A$: The power iteration is 
\begin{align*}
v^{k+1} = \frac{V^{*}Av^{k}}{\norm{V^{*}Av^{k}}}
\end{align*}
and there are different choices to compute the estimate the norm. Textbooks often state that 
\begin{align}\label{eq:normA1}
\norm{V^{*}A}\approx \scp{v^{k}}{V^{*}Av^{k}},\quad \textnormal{i.e.}\quad\norm{A}\approx \sqrt{\scp{v^{k}}{V^{*}Av^{k}}}
\end{align}
should be used. The implementations of \texttt{normest} in R or Octave (for general matrices) use 
\begin{align}\label{eq:normA2}
\norm{A}\approx \frac{\norm{V^{*}Av^{k}}}{\norm{Av^{k}}}
\end{align}
and another sensible choice is 
\begin{align}\label{eq:normA3}
\norm{V^{*}A} \approx \norm{V^{*}Av^{k}},\quad \textnormal{i.e.}\quad \norm{A}\approx \sqrt{\norm{V^{*}Av^{k}}}
\end{align}
and finally, we could also use 
\begin{align}\label{eq:normA4}
\norm{A}\approx \norm{Av^{k}}.
\end{align}
While these values do all converge to the same value when $V^{*} = A^{*}$ is the exact adjoint, this is not the case anymore for inexact adjoints.
The iteration will still converge when  $V^{*}A$ has a unique dominant eigenvalue (i.e. one of largest margnitude).

\begin{table}[t]
  \centering
  \begin{tabular}{rrrr}
    \toprule
    $\norm{A}$ from~\eqref{eq:normA1} & $\norm{A}$ from~\eqref{eq:normA2} & $\norm{A}$ from~\eqref{eq:normA3} & $\norm{A}$ from~\eqref{eq:normA4} \\\midrule
     $8.2216$ & $1.2326$  & $8.2216$  & $54.8402$\\
    \bottomrule
  \end{tabular} 
  \caption{Numerical approximation of the operator norm of the Radon transform in \texttt{scikit-image} by the power method and different choices of the approximation.}
  \label{tab:table-norm-radon}
\end{table}

We apply the power iteration method to compute the norm of the linear operator in the previously mentioned example of the Radon transform implemented from the Python package \texttt{scikit-image} \cite{scikit-image} in the same setting as described above. We compute the norm using four different methods, which involve the adjoint operator in different ways. The results, reported in Table~\ref{tab:table-norm-radon}, show that the power iteration method is sensitive to errors in the adjoint operator, leading to notably different values of the norm.
This highlights the importance of accurately implementing the adjoint operator when using the power iteration. However, a method that only requires the evaluation of the operator itself could provide a more robust and reliable way to compute the norm and help to clarify the effects of errors in the adjoint operator on the power iteration. We will come back to this issue in a numerical example in Section~\ref{sec:implementation-radon}.

\subsection{Related work}
\label{sec:related-work}

Our restrictive assumptions rule out existing methods:
The storage requirement rules out the brute force approach of building a full matrix representation for the linear map by evaluating $Ae_{i}$ for all $i$ and then calculating the singular value decomposition of this representation as this requires storagy $\bigO(md)$ which can be considerably more than $\bigO(\max(m,d))$.
Methods based on sketching of the map, as done in~\cite{martinsson2002randomized} for the Schatten $p$-norm for even $p$ or in \cite{magdon2011note} for the spectral norm (which is the Schatten $\infty$-norm),
lead to lower bounds with probabilistic guarantees, 
but these depend on the size of the sketch. Further improvements of the estimate would need a refinement of the sketch and another computation of the norm of the sketched matrix.
In \cite{kong2017spectralestimation} the estimation of the spectrum of a data set is considered,
which includes the estimation of the operator norm as Schatten $\infty$-norm also as limit of the $k$-th moments.
From the main theorem (cf.~\cite[Thm.~1]{kong2017spectralestimation})
we would need the limit $k\to\infty$ of the moments,
where the estimated bound behaves super-exponential in $k$.
Choosing the moment just large enough might be not effective 
and depends on the actual values of the singular values.
The work \cite{li2014sketching} also considers the estimation of the Schatten $\infty$-norm by linear and bilinear sketching.
Crucially, they show that any approximation of the spectral norm needs a sketch that is basically almost as large as the full matrix, i.e. of order $\Omega(n^{2})$, where $n$ is the minimum of its number of rows and columns.
Another popular method for the computation of the spectral norm is the power iteration~\cite{kuczynski1992estimating} which, in its plain form, works for maps for which the input and output dimensions are equal.
To calculate the norm of a map with different input and output dimensions, one could in principle apply the power method to the map $A^{*}A$, however,
this needs that a routine for $A^{*}$ is available, which we do not assume hIere.
If the evaluation of $A$ is done by a computer program that has been written in a way such that automatic differentiation is applicable,
we could simply use that to calculate the derivative of the map $v\mapsto \norm{Av}^{2}$
(and we could even calculate $A^{*}w$ for some $w$ by applying automatic differentiation with respect to $v$ to the function  $v\mapsto \scp{w}{Av}$).
However, this is not always the case, especially for precompiled code.
Moreover, it is known that the power method converges slowly if the spectral gap (i.e. the distance between the two largest eigenvalues) is small.
Other approaches to operator norms such as in \cite{kenney1998} (focusing on condition estimation) and \cite{dudley2022} (using Monte Carlo techniques for Schatten-$p$ norms) also only consider maps with equal input and output dimension.
Methods in numerical linear algebra that avoid using adjoints (or transposes) have been studied since at least the 1990s and are known as ``adjoint-free (or transpose-free) methods'', see~\cite{chan1998transpose} for a motivation to consider such methods.
Some examples are the transpose-free quasi minimal residual method (TFQMR) from~\cite{freund1993tfqmr} and Lanczos type methods \cite{brezinski1998transpose,chan1998transpose}  which lead to the \emph{conjugate gradient squared method} (CGS)~\cite{sonneveld1989cgs}  for the solution of linear systems, and for least squares problems there is the method of random descent from~\cite{lorenz2023randomdescent}.
The recent paper \cite{boulle2024operator} treated the problem of how to learn a linear map if only evaluations of the map are available and no queries of the adjoint are available.

\subsection{Contribution}
\label{sec:contribution}

The contribution of this work is:
\begin{itemize}
\item We propose a stochastic algorithm that produces a sequence that converges to the operator norm 
\begin{align*}
    \max_{\norm{v}=1} \norm{A v}^2
\end{align*}
almost surely only using evaluations of the operator $A$, and calculation of norms and inner products.
\item Our method works for any linear map between finite dimensional Hilbert spaces, i.e. we do not need any assumptions on the dimension of the spaces or on the structure of the linear map.
\item The method has the minimal possible storage requirement $\bigO(\max(m,d))$, i.e., we only need to store two vectors of the input dimension $d$ and two vectors of the output dimension $m$ during the iteration.
\end{itemize}
We focus on the case where $A$ maps $\RR^{d}$ to $\RR^{m}$ and we work with the standard Euclidean inner product and norm. However, what we develop can be generalized to more general finite dimensional Hilbert spaces as well as to spaces over the complex numbers.

Additionally, our method produces an approximation of a right singular vector for the largest singular value and can be extended to compute further right singular vectors. Curiously, the method also allows us to check if a map $A$ is orthogonal in the sense that $A^{*}A = cI$ for some constant $c \geq 0$, where $I$ denotes the identity on the input space of $A$.  

\subsection{Notation}
\label{sec:notation}

We $\norm{x}$ and $\scp{x}{y}$ for the standard Euclidean vector norm and inner product, respectively,  $\norm{A}$ for the induced operator norm and $I_{d}$ for the identity map on $\RR^{d}$. With $\cN(\mu,\Sigma)$ we denote the normal distribution with mean $\mu$ and covariance $\Sigma$ and with $\sphere^{d-1} \coloneqq \{x\in\RR^{d}\mid \norm{x}=1\}$ we denote the unit sphere in $d$-dimensional space. With $\PP(B)$ we denote the probability of an event $B$ (and the underlying distribution will be clear from the context).
In this paper $\sigma_{k}$ will always denote the $k$-th largest singular value of the linear map induced by $A$, i.e. $\sigma_{1} = \sigma_{\max} = \norm{A}$. For a vector $v \in \RR^d$, we denote by $\{v\}^{\bot} \coloneqq \left\{ w \in \RR^d \mid \scp{w}{v}=0 \right\}$ the hyperplane orthogonal to $v$.
By $A^*$,
we denote the adjoint operator of $A$ defined by $\scp{x}{A^{*}y} = \scp{Ax}{y}$.
Similarly, we use $x^*$ for the induced linear map on the dual space (which is the transpose of the vector). We will occasionally abbreviate ``almost surely'' with a.s. which means that the described event will occur with probability one. By $\mathbf{1}$, we denote a vector of all ones (and the length will be clear from the context) and $\mathbbm{1}_{B}$ is used for the characteristic function of the set $B$.

\section{Adjoint free calculation of $\norm{A}$}
\label{sec:adjoint-free-A}

We propose a stochastic method to approximate the operator norm $\norm{A}$ for some $A$ that maps from $\RR^{d}$ to $\RR^{m}$. The method relies on the Rayleigh principle namely $\norm{A}^{2} = \max_{\norm{v}=1}\norm{Av}^{2}$. Within our assumptions, we cannot apply a simple projected gradient ascent for this problem as this would require to compute the vector $A^{*}Av$ and we do not have access to the adjoint.

The method we propose is a random search method and operates fairly simple:
\begin{enumerate}
\item Initialize some $v^{0}\in\RR^{d}$ with $\norm{v^{0}}=1$ and set $k=0$.
\item Generate $y^{k}\sim \cN(0,I_{d})$.
\item Project $y^{k}$ onto the space orthogonal to $v^k$ and normalize it 
  \begin{align*}
    x^{k} = \frac{y^{k} - \scp{y^{k}}{v^{k}}v^{k}}{\norm{y^{k} - \scp{y^{k}}{v^{k}}v^{k}}}
  \end{align*}
  to get a random normalized search direction perpendicular to $v^{k}$.\footnote{This step would break down if $y^{k}$ and $v^{k}$ are parallel---however, this happens only with probability zero.}
\item Calculate a stepsize $\tau_{k}$ by\footnote{It may happen that there exists no maximizer. However, we will argue that this case occurs with probability zero, see Proposition~\ref{prop:stepsize-formula} and Corollary~\ref{cor:no-sing-vectors-during-iteration}.} 
  \begin{align}\label{eq:argmax-tau}
    \tau_{k} 
    \in
    \argmax_{\tau \in \RR} 
    \frac{\norm{Av^{k} + \tau Ax^{k}}^{2}}{\norm{v^{k}+\tau x^{k}}^{2}},
  \end{align}
\item Update $v^{k+1} = \frac{v^{k}+\tau_kx^{k}}{\norm{v^{k}+\tau_{k}x^{k}}}$, increment $k$ and go to step 2. 
\end{enumerate}

\begin{remark}[Normalization of $v^{k}$ and $x^{k}$]
    \label{rem:norm_const}
  By design we always have that $\norm{v^{k}} = \norm{x^{k}}=1$ and $\scp{v^{k}}{x^{k}} = 0$.
\end{remark}

\begin{remark}[Critical points]
  By Rayleigh's principle, the critical points of the problem $\max_{\norm{v}=1}\norm{Av}$ are exactly the right singular vectors of $A$ and the values at these points are exactly the singular values.
  More precisely, the intersection of the eigenspace of the largest (or smallest) singular value gives exactly the global maxima (or minima). All other eigenspaces lead to saddle points.
  We will show in the course of the paper that the method we propose will (almost surely) avoid all saddle points and will converge (almost surely) to the global maximum.
\end{remark}

A particularity of this method is that the search direction is quite arbitrary, 
but the stepsize is determined exactly (and note that it is allowed to be negative).
This is different from many other stochastic methods which estimate update directions and then take a step according to a fixed stepsize schedule.

In the following, we will show that this method is simple to execute (and especially discuss how~\eqref{eq:argmax-tau} is solved in practice) and indeed uses only evaluations of $A$.
We will also show that the method will produce a sequence $(v^{k})_{k \in \NN}$ such that $\norm{Av^{k}}$ converges to $\norm{A}$ almost surely, and that the distance of $(v^k)_{k \in \NN}$ to the right singular space of $A$ corresponding to the leading singular value converges to zero.

\begin{remark}[Distribution of the $x^{k}$]\label{rem:distribution-x}
  The random vector $x^{k}$ is distributed as follows:
  Since an orthogonal projection of normally distributed random vectors is again normally distributed, the projection $y^{k } - \scp{y^{k}}{v^{k}}v^{k}$ is normally distributed on the tangent space of the unit sphere $\sphere^{d-1}$ at $v^{k}$. 
  Since $x^{k}$ is the result of the normalization of this vector, $x^{k}$ is uniformly distributed on the unit sphere in this tangent space, i.e. on a rotated version of $\sphere^{d-2}$.
\end{remark}

\subsection{Details of the algorithm}
\label{sec:details}

We start with a discussion for the computation of the stepsize $\tau_{k}$ from \eqref{eq:argmax-tau}.
Utilizing Remark~\ref{rem:norm_const}, we observe that 
$\norm{v^{k}+\tau x^{k}}^{2} 
= \norm{v^{k}}^{2} + 2\tau\scp{v^{k}}{x^{k}} + \tau^{2}\norm{x^{k}}^{2} = 1+\tau^{2}$
and hence, 
the objective in the minimization problem~\eqref{eq:argmax-tau} in the $k$th iteration reads as 
\begin{align}\label{eq:h}
  h_k(\tau) \coloneqq \frac{\norm{Av^{k} + \tau Ax^{k}}^{2}}{1+\tau^{2}}
\end{align}
 Its derivative is 
  \begin{align*}
    h_k'(\tau) 
    & = \frac{(2\scp{Av^{k}}{Ax^{k}} + 2\tau\norm{Ax^{k}}^{2})(1+\tau^{2})}{(1+\tau^{2})^{2}} \\
    & \qquad - \frac{(\norm{Av^{k}}^{2} + 2\tau\scp{Av^{k}}{Ax^{k}} + \tau^{2}\norm{Ax^{k}}^{2})\cdot2\tau}{(1+\tau^{2})^{2}}\\
    & = 2\frac{\scp{Av^{k}}{Ax^{k}} + (\norm{Ax^{k}}^{2}- \norm{Av^{k}}^{2})\tau - \scp{Av^{k}}{Ax^{k}}\tau^{2}}{(1+\tau^{2})^{2}}\\
    & = 2\frac{a_{k} + b_{k}\tau - a_{k}\tau^{2}}{(1+\tau^{2})^{2}}
  \end{align*}
where we defined 
\begin{align}
\label{eq:def-ak-bk}
a_{k} \coloneqq \scp{Av^{k}}{Ax^{k}}\quad \text{and}\quad b_{k}\coloneqq \norm{Ax^{k}}^{2} - \norm{Av^{k}}^{2}.
\end{align}
The shape of the objective $h_k$ depends on the signs of $a_{k}$ and $b_{k}$ and we illustrate the cases in Figure~\ref{fig:shape-h}.
It is clear that 
\begin{align*}
  h_k(0) = \norm{Av^{k}}^{2}\quad \text{and}\quad
\lim_{\tau\to\pm\infty} h_k(\tau) = \norm{Ax^{k}}^{2}.
\end{align*}

\pgfplotsset{width=5cm,height=4cm,
  xlabel style={xlabel=$\tau$},
  samples = 50,
  every axis plot post/.append style={orange,mark=none},
  every axis post/.append style={xmin=-5, xmax=5, ymin=0, ymax = 1.9, axis lines=middle, xtick=\empty},
  every tick label/.append style={font=\tiny}
}

\begin{figure}[htb]
  \centering
  \begin{tabular}{ccc}
    \begin{tikzpicture}
      \begin{axis}[ytick={1,1.5},yticklabels={$\norm{Av^{k}}^{2}$,$\norm{Ax^{k}}^{2}$}]
        \addplot{(1+x*1+1.5*x*x)/(1+x*x)};
        \addplot[thin,dotted] {1.5};
      \end{axis}
    \end{tikzpicture}&
    \begin{tikzpicture}
      \begin{axis}[ytick={1,1.5},yticklabels={$\norm{Av^{k}}^{2}$,$\norm{Ax^{k}}^{2}$}]
        \addplot{(1-x*1+1.5*x*x)/(1+x*x)};
        \addplot[thin,dotted] {1.5};
      \end{axis}
    \end{tikzpicture}&
    \begin{tikzpicture}
      \begin{axis}[ytick={1,1.5},yticklabels={$\norm{Av^{k}}^{2}$,$\norm{Ax^{k}}^{2}$}]
        \addplot{(1+1.5*x*x)/(1+x*x)};
        \addplot[thin,dotted] {1.5};
      \end{axis}
    \end{tikzpicture}\\\bigskip
    $a_{k}>0$, $b_{k}>0$ &
    $a_{k}<0$, $b_{k}>0$ &
    $a_{k}=0$, $b_{k}>0$\\
    \begin{tikzpicture}
      \begin{axis}[ytick={1,1.5},yticklabels={$\norm{Ax^{k}}^{2}$,$\norm{Av^{k}}^{2}$}]
        \addplot{(1.5+x*1+1*x*x)/(1+x*x)};
        \addplot[thin,dotted] {1};
      \end{axis}
    \end{tikzpicture}&
    \begin{tikzpicture}
      \begin{axis}[ytick={1,1.5},yticklabels={$\norm{Ax^{k}}^{2}$,$\norm{Av^{k}}^{2}$}]
        \addplot{(1.5-x*1+1*x*x)/(1+x*x)};
        \addplot[thin,dotted] {1};
      \end{axis}
    \end{tikzpicture}&
    \begin{tikzpicture}
      \begin{axis}[ytick={1,1.5},yticklabels={$\norm{Ax^{k}}^{2}$,$\norm{Av^{k}}^{2}$}]
        \addplot{(1.5+1*x*x)/(1+x*x)};
        \addplot[thin,dotted] {1};
      \end{axis}
    \end{tikzpicture}\\
    $a_{k}>0$, $b_{k}<0$ &
    $a_{k}<0$, $b_{k}<0$ &
    $a_{k}=0$, $b_{k}<0$\\
  \end{tabular}
  \caption{Shape of $h_k$ from~\eqref{eq:h} for different cases of the signs of $a_k$ and $b_k$.}
  \label{fig:shape-h}
\end{figure}

\begin{proposition}
  \label{prop:stepsize-formula}
    For $h_k$ defined by~\eqref{eq:h} and $a_{k}$ and $b_{k}$  defined by~\eqref{eq:def-ak-bk} it holds that:
  \begin{itemize}
  \item[i)] If $a_k \neq 0$,
    then the $\tau_{k}$ that solves~\eqref{eq:argmax-tau} is unique and given by
    \begin{align*}
      \tau_{k}
      = 
      \sign(a_k) \left(\tfrac{b_{k}}{2\abs{a_{k}}} + \sqrt{\tfrac{b_{k}^{2}}{4a_{k}^{2}}+1}\right).
    \end{align*}
  \item[ii)] 
    If $a_k = 0$, and $b_{k}<0$, then $h_k$ attains its maximum at $\tau = 0$
  \item[iii)] If $a_k=0$ and  $b_{k}>0$, then $h_{k}$ does not attain its supremum.
  \item[iv)] If $a_k=0$ and $b_{k} = 0$, then $h$ is constant.
  \end{itemize}
\end{proposition}
\begin{proof}
  The local extrema of $h_k$ are at the roots of the enumerator of $h_k'$, which are 
  \begin{align*}
    \tau_{\pm} 
    \coloneqq 
    \tfrac{b_{k}}{2a_{k}} \pm \sqrt{\tfrac{b_k^2}{4a_k^2}+1} 
  \end{align*}
  and it holds that $\tau_{+}>0$ and $\tau_{-}<0$. These are the only local extrema.

  If $a_k = \scp{Av^{k}}{Ax^{k}}>0$ we have $h_k'(0)>0$ and the local maximum is on the positive real line. Hence, $\tau_{+}$
  is the global maximizer of $h_k$, since the sign of $h_k'$ is defined by the quadratic function $a_k + b_k\tau - a_k\tau^2$ such that $h_k'(\tau) < 0$ for all $\tau > \tau_+$.

  If $a_k < 0$, we have $h_k'(0)<0$ and one can check that
  the global maximum of $h_k$ is attained at
  \begin{equation*}
    \tau_{-} 
    = 
    \tfrac{b_{k}}{2a_{k}} - \sqrt{\tfrac{b_k^2}{4a_k^2}+1} 
    = - \left( \tfrac{b_{k}}{2 |a_{k}|} + \sqrt{\tfrac{b_k^2}{4 a_k^2}+1} \right).
  \end{equation*}
  If $a_k = 0$, we have
  \[
    h_k(\tau) = \frac{\norm{A v^k}^2 + \tau^2 \norm{A x^k}^2}{1 + \tau^2}
    = \frac{\norm{A v^k}^2 - \norm{A x^k}^2}{1 + \tau^2} + \norm{A x^k}^2
    = -\frac{b_k}{1 + \tau^2} + \norm{A x^k}^2
  \]
  and hence the sign of $b_k$ yields the different cases.
\end{proof}

With the previous result concerning the optimal stepsize, 
we give the full algorithm in detail as Algorithm~\ref{alg:mafno-orth-exact}.
Note that the algorithm would break down
if $a_{k}=0$ since we would have to divide by zero.
However, if $a_k = 0$ and $b_k > 0$ hold,
then any $\tau \neq 0$ would lead to a strictly larger objective value,
cf.~Figure~\ref{fig:shape-h} (upper right block).
We will show that this will almost surely not occur, 
but before we do so, 
we provide a few remarks regarding the algorithm.

\begin{algorithm}[htb]
  \caption{Matrix- and adjoint free operator norms}\label{alg:mafno-orth-exact}
  \begin{algorithmic}[1]
    \State Initialize $v^{0}\in\RR^{d}$ with $\norm{v^{0}}=1$ by $v^0 = \hat v /\norm{\hat v}$, where $\hat v \sim \cN(0,I_{d})$
    \For{$k=0,1,2,\dots$}
      \State Sample $y^{k}\sim \cN(0,I_{d})$.
      \State Project
      \begin{align*}
        x^{k} = \frac{y^{k} - \scp{y^{k}}{v^{k}}v^{k}}{\norm{y^{k} - \scp{y^{k}}{v^{k}}v^{k}}}.
      \end{align*}
      \State Calculate
      \begin{align*}
        a_{k} & = \scp{Av^{k}}{Ax^{k}},&   b_{k} & = \norm{Ax^k}^2 - \norm{Av^{k}}^2.
      \end{align*}
      \State Calculate stepsize
      \begin{align*}
        \tau_{k} = \sign(a_{k})\left(\tfrac{b_{k}}{2\abs{a_{k}}} + \sqrt{\tfrac{b_{k}^{2}}{4a_{k}^{2}}+1}\right).
      \end{align*}
      \State Update
      \begin{align}
        \label{eq:v-update-step}
        v^{k+1}
        = 
        \frac{v^{k} + \tau_{k}x^{k}}{\sqrt{1+\tau_{k}^{2}}}.
      \end{align}
    \EndFor
    \State \textbf{return} estimate of $\norm{A}$:
    \[
      \norm{Av^{k}}
    \]
  \end{algorithmic}
\end{algorithm}

\begin{remark}[Convergence of the sequence $\norm{Av^{k}}$]\label{rem:conv-normAv}
  As long as $a_{k}\neq0$, the algorithm leads to a sequence $\norm{Av^{k}}$ which is strictly increasing. Since we have $\norm{Av^{k}}\leq \norm{A}$, the sequence $\norm{Av^{k}}$ is convergent as long as we have $a_{k}\neq 0$ throughout the iteration.
\end{remark}

\begin{remark}[Algorithm~\ref{alg:mafno-orth-exact} is a stochastic projected gradient method]
  Although we sample $x^{k}$ uniformly from the unit directions orthogonal to $v^{k}$, the method is still a stochastic gradient method in some sense. This is because the stepsize $\tau_{k}$ is positive for $a_k = \scp{Av^k}{Ax^k}>0$ and negative for $a_{k}<0$, and thus, we effectively move in direction $\sign(a_{k})x^{k}$. This means that the expected direction in which the algorithm moves is uniformly distributed on a half-sphere in the space orthogonal to $v^{k}$.
  This half-sphere is defined by the inequality $a_{k}>0$ which is $0<\scp{Av^{k}}{Ax^{k}} = \scp{A^{*}Av^{k}}{x^{k}}$. Due to the symmetry of the distribution, the expected direction $\EE(\sign(a_{k})x^{k})$ is a positive multiple of the projection of $A^{*}Av^{k}$ on the space orthogonal to $v^{k}$. We conclude that Algorithm~\ref{alg:mafno-orth-exact} is a stochastic ascent algorithm in the sense that the expected direction is a multiple of the projection on the gradient $2A^{*}Av^{k}$ of the objective $\norm{Av}^{2}$.
\end{remark}

\begin{remark}[Markov chain provided by $(v^k)_{k \in \NN}$]
    \label{rem:markov_chain}
    The random variables $(v^k)_{k \in \NN}$ generated by Algorithm~\ref{alg:mafno-orth-exact}
    can be understood as a sequence of measurable functions 
    $v^k : \Omega \to \sphere^{d-1}$.
    Here, $\Omega$ is a sampling space equipped with a sigma algebra $\cF$ and probability measure $\PP$
    forming a probability space $(\Omega, \cF, \PP)$.
    More precisely,
    the sequence $(v^k)_{k \in \NN}$ is a Markov-chain.
    The sampling space $\Omega$ is given by $(\sphere^{d-1})^{\NN}$
    and $\cF = \cS^{\otimes \NN}$.
    The probability measure $\PP$
    is then defined via the Markov chain rule using the defined Markov kernel equipped with the update \eqref{eq:v-update-step}
    and closed formula for $\tau_k$ from Proposition~\ref{prop:stepsize-formula}.
\end{remark}

Now
we show that $a_{k}=0$ only occurs if $v^{k}$ is a right singular vector
and that this will (almost surely) never occur during the algorithm.

\begin{lemma}
\label{lem:a=0-at-sing-vecs}
  If $v^{k}$ is not a right singular vector of $A$ and $x^{k}$ is sampled according to step 3 and 4 of Algorithm~\ref{alg:mafno-orth-exact}, then $a_{k} = \scp{Av^{k}}{Ax^{k}} \neq 0$ almost surely.
\end{lemma}
\begin{proof}
  By construction it holds that $\scp{x^{k}}{v^{k}}=0$ and if we also have that $a_{k}=0$, then it holds that $0 = \scp{Av^{k}}{Ax^{k}} = \scp{A^{*}Av^{k}}{x^{k}}$. Hence, $a_{k}=0$ only holds if $x^{k}$ is orthogonal to $v^{k}$ and to $A^{*}Av^{k}$ and since $v^{k}$ is not an eigenvector of $A^{*}A$, this means that the vector $x^{k}$ is in a subspace of dimension at most $d-2$. Recalling Remark~\ref{rem:distribution-x} on the distribution of $x^{k}$, we see this restricts $x^{k}$ to a subset of measure zero on the unit sphere in the $d-1$-dimensional tangent space of $\sphere^{d-1}$ at $v^{k}$ and hence, this happens with probability zero. 
\end{proof}

By the previous lemma,
we know that $a_{k}=0$ only happens (a.s.) 
if $v^{k}$ is a right singular vector. 
Now we show that if $v^{k}$ is not a right singular vector, 
the same is true (a.s.) for $v^{k+1}$.

We start with the following observation regarding the eigenspaces of $A^*A$ for some eigenvalue $\sigma^{2}$, i.e. singular value $\sigma$ of $A$, denoted by 
\[
    E_\sigma = \left\{ x \in \RR^d \mid A^{*}Ax = \sigma^2 x\right\}.
\]

\begin{lemma}
    \label{lem:dimension-directions}
  Let $A\in\RR^{m\times d}$, $\sigma^2$ be an eigenvalue of $A^{*}A$ with multiplicity $r<d$, $v \in \sphere^{d-1}$ be no eigenvector for the eigenvalue $\sigma^{2}$. Then it holds that the set of vectors $x\in \{v\}^{\bot}$ for which $v+x\in E_{\sigma}$ holds is an affine subspace of dimension $r$ or $r-1$, or is empty.
\end{lemma}
\begin{proof}
  Since $\dim(E_{\sigma}) = r$ we can write $E_{\sigma} = \left\{ x \mid Cx = 0 \right\}$ with $C\in \RR^{d-r\times d}$ having rank $d-r$. We assume without loss of generality (by applying a suitable rotation) that $v = e_{1} = (1,0,\dots,0)^{*}$, hence $x\in \{v\}^{\bot}$ means $x = (0,x')$ with $x' \in \RR^{d-1}$. We have that $v+x\in E_{\sigma}$ exactly if $C(v+x)=0$ which is equivalent to $Cx = -Cv$. Denoting $C = \left[ c_1, C'  \right]$ with $c_{1}\in\RR^{d-r}$ and $C' \in \RR^{d-r\times d-1}$, i.e. $c_{1}$ is the first column of $C$ and $C'$ the matrix consisting of the remaining columns, we can rewrite the latter system of equations as $C' x'  = -c_{1}$. The matrix $C'$ has rank either $d-r$ or $d-r-1$ as cutting the first column might reduce the dimension of the span of those. Hence, this system has an affine solution space of dimension either $d-1 - (d-r) = r-1$ or $d-1 - (d-r-1) = r$, or might have no solution at all.
\end{proof}

\begin{corollary}
    \label{cor:no-sing-vectors-during-iteration}
  Let $d > 2$ and $A\in\RR^{m\times d}$ and assume that all eigenvalues of $A^{*}A$ have multiplicities less than $d-1$. Then it holds that if $v^{k}$ is not a right singular vector of $A$, then $v^{k+1}$ is almost surely also not a right singular vector of $A$.

  Consequently, we have $a_{k}\neq 0$ (almost surely) throughout the iteration of Algorithm~\ref{alg:mafno-orth-exact} if $v^{0}$ is not a right singular vector of $A$ and the sequence $\norm{Av^{k}}$ is increasing and convergent (a.s.).
\end{corollary}
\begin{proof}
  Let $\sigma$ be a singular value of $A$ with multiplicity $r$ and $E$ the respective eigenspace.
  
  If $v^{k}+\tau x^{k}$ is in $E$, then (by Lemma~\ref{lem:dimension-directions}) $\hat{x} = \tau x^{k}$ is from an affine subspace $L$ of dimension at most $r$.
  Hence, the set of directions $x^{k} \in \{v^k\}^\perp \cap \sphere^{d-1}$ that may result in $v^{k+1}\in E$ is a subset of the set $\left\{x/\norm{x}  \mid x\in L \right\}$. Since $v^{k}$ and $x^{k}$ are not parallel, this is a set of probability zero (with respect to the Haar measure on $\{v^k\}^\perp \cap \sphere^{d-1}$, according to which $x^{k}$ is sampled due to Remark~\ref{rem:distribution-x}), if $r<d-1$ (which then also implies $r-1<d$). The latter (restrictive case from Lemma~\ref{lem:dimension-directions}) is fulfilled by assumption. Since this is true for all (finitely many) eigenspaces of $A^{*}A$, the claim follows.

  The monotonicity and the convergence of the sequence $\norm{Av^{k}}$ 
  is a consequence of Remark~\ref{rem:conv-normAv} under the assumption.
\end{proof}

Corollary~\ref{cor:no-sing-vectors-during-iteration} guarantees that Algorithm~\ref{alg:mafno-orth-exact} will not run into division by $a_{k}=0$ almost surely 
if all singular values of $A$ have multiplicity at most $d-2$, (which can be guaranteed, for example, if $A$ has three different singular values).
In the case $d=1$, $A$ is just a vector and $A^{*}A$ is just a scalar, so nothing is to be done. If $d = 2$ and $A$ has two different singular values,
then a simple geometric argument shows that our algorithm 
converges exactly after one step,
see Section~\ref{sec:numerical-experiments} and Fig.~\ref{fig:circle-convergence-algorithm}.
Hence, 
we make the following assumption and assume that it is fulfilled from hereon as long as nothing else is said.

\begin{assumption}
    \label{ass:ass_sing_rang}
  The operator $A$ from $\RR^{d}$ to $\RR^{m}$ has no singular value with multiplicity $d-1$ or $d$.
\end{assumption}

If we sample $v^{0}$ uniformly on $\sphere^{d-1}$, Assumption~\ref{ass:ass_sing_rang} guarantees that $a_{0}\neq 0$ almost surely and by Corollary~\ref{cor:no-sing-vectors-during-iteration} this is preserved throughout the iteration.
Later in Remarks~\ref{rem:multiplicity-d-1} and~\ref{rem:second_assump_gone} we will see that we can even drop Assumption~\ref{ass:ass_sing_rang};
hence, Algorithm~\ref{alg:mafno-orth-exact} can be applied to any linear map $A$ mapping from $\RR^d$ to $\RR^m$.

 \subsection{Almost sure convergence}
\label{sec:convergence}

Now we set out to prove that Algorithm~\ref{alg:mafno-orth-exact} will almost surely give a sequence $(v^{k})_{k \in \NN}$ such that $\norm{Av^{k}}\to\norm{A}$. 
The first step is to calculate the ascent in each iteration.
Secondly,
we discuss the convergence of $(a_k)_{k \in \NN}$.
Afterwards,
we can prove the convergence of the sequence $(\norm{A v^k}^2)$ to a singular value
and that at least one subsequence $(v^{k_j})_{k \in \NN}$ converges into the corresponding singular space (almost surely).
Finally,
the almost sure convergence towards the maximal singular value for $(\norm{A v^k})_{k \in \NN}$
is proven by considering a technical lemma, 
which allows us to uniformly bound the probability of finding optimal search directions
in all tangential hyperplanes.

\begin{lemma}
  \label{lem:ascent-formula}
  For the random variables $(v^{k})_{k \in \NN}, (a_{k})_{k \in \NN}$ 
    and $(\tau_{k})_{k \in \NN}$ from Algorithm~\ref{alg:mafno-orth-exact} it holds 
  \begin{align*}
    \norm{Av^{k+1}}^{2} - \norm{Av^{k}}^{2} 
    = 
    \tau_{k}a_{k}, 
    \qquad \forall k \in \NN,
    \quad \textrm{a.s.}
  \end{align*}
\end{lemma}
\begin{proof}
  We calculate the ascent in the objective: 
  \begin{align*}
    \norm{A v^{k+1}}^2 - \norm{A v^{k}}^2 & = \tfrac1{1+\tau_k^2}\norm{Av^{k} + \tau_k Ax^{k}}^2 - \norm{Av^{k}}^2\\
                                          & = \tfrac1{1+\tau_k^2}\left[\norm{Av^{k}}^{2} + 2\tau_k\scp{Av^{k}}{Ax^{k}} + \tau_k^2\norm{Ax^{k}}^2\right]- \norm{Av^{k}}^2\\
                                          & = \tfrac{2\tau_k}{1+\tau_k^{2}}\scp{Av^{k}}{Ax^{k}} + \tfrac{\tau_k^{2}}{1+\tau_k^{2}}\norm{Ax^{k}}^{2} + \left( \tfrac1{1+\tau_k^{2}} -1\right)\norm{Av^{k}}^{2}\\
                                          & = \tfrac{2\tau_k}{1+\tau_k^{2}}\scp{Av^{k}}{Ax^{k}} + \tfrac{\tau_k^{2}}{1+\tau_k^{2}}\left( \norm{Ax^{k}}^{2} - \norm{Av^{k}}^{2} \right)\\
                                          & = \tfrac{2\tau_k}{1+\tau_k^{2}}a_{k} + \tfrac{\tau_k^{2}}{1+\tau_k^{2}} b_{k}\\
                                          & = \tfrac{\tau_k}{1+\tau_k^{2}}\left( a_{k} + (a_{k} + b_{k}\tau_k) \right).
  \end{align*}
  Since we know that $\tau_k$ solves $-a_{k}\tau_k^{2} + b_{k}\tau_k + a_{k} = 0$ we replace $a_{k} + b_{k}\tau_k = a_{k}\tau_k^{2}$ and get 
  \begin{equation*}
    \norm{Av^{k+1}}^{2} - \norm{Av^{k}}^{2} = \tfrac{\tau_k}{1+\tau_k^{2}}(a_{k} + \tau_k^{2}a_{k}) = \tau_k a_{k}.
    \qedhere
  \end{equation*}
\end{proof}

Now we show that the sequence $(a_{k})_{k \in \NN}$ converges to zero almost surely.
\begin{lemma}
  \label{lem:a-to-0}
  For the random variables $(v^{k})_{k \in \NN}, (x^{k})_{k \in \NN}$ 
  and $(a_{k})_{k \in \NN}$ generated by Algorithm~\ref{alg:mafno-orth-exact}
  it holds that $a_{k}=\scp{Av^{k}}{Ax^{k}} \to 0$ almost surely as $k \to \infty$.
\end{lemma}
\begin{proof}
  First observe that since $\norm{Av^{k}}^{2}$ is convergent by Corollary~\ref{cor:no-sing-vectors-during-iteration}, Lemma~\ref{lem:ascent-formula} proves that $\tau_{k}a_{k}\to 0$.
  Second, we have by definition that
  $a_{k} = a_{k}\tau_k^{2} - b_{k}\tau_{k} = \tau_{k}(a_{k}\tau_{k} - b_{k})$. 
  We multiply by $a_{k}$ 
  and get 
  \begin{align}
    \label{eq:ab_tau}
    a_k\tau_k(a_{k}\tau_{k}- b_k) = a_k^{2},
    \qquad k \in \NN, \quad \textrm{a.s.}
  \end{align}
  Since $\abs{b_{k}} \leq \max\{\norm{A v^k}, \norm{A x^k}\} \leq \norm{A}$ is bounded, 
  the left-hand side vanishes as $k \to \infty$, which proves the claim.
\end{proof}

\begin{lemma}
\label{lem:E_a_k_sqr}
    It holds that
    \begin{align*}
        \EE\big[ a_k^2 \mid v^k \big] = \frac{1}{d-1}\big\| \big(I_d-v^k(v^k)^*\big) A^*Av^k \big\|^2.
    \end{align*}
\end{lemma}
\begin{proof}
By symmetry arguments (see also Remark~\ref{rem:distribution-x}) we have that 
\begin{align*} 
\EE[x^k(x^k)^* \mid v^k] = \frac{1}{d-1}\big(I_d-v^k(v^k)^*\big).
\end{align*}
It follows that
\begin{align*}
    \EE\big[ a_k^2 \mid v^k \big] 
    &= \EE\big[ \scp{ Av^k}{Ax^k}^2 \mid v^k \big] \\
    &= \EE\big[ (v^k)^*A^*Ax^k(x^k)^*A^*Av^k \mid v^k \big] \\
    &= (v^k)^* A^*A \cdot \EE[x^k(x^k)^* \mid v^k] \cdot A^*Av^k
\end{align*}
and the assertion follows,
since $I_d - v^k(v^k)^*$ is idempotent.
\end{proof}

Next, we show the existence of a subsequence of the random variables $(v^k)_{k \in \NN}$
generated by Algorithm~\ref{alg:mafno-orth-exact}
converging into a singular space, almost surely.
Crucially,
we do not have the existence of accumulation points of the sequence $(v^k)_{k \in \NN}$
of random variables in general.
Therefore,
we define the distance of a set $C$ and a point by
\[
    \dist(C, x) \coloneqq \inf_{y \in C} \; \norm{x - y}.
\]
Recall that the infimum is attained by some $y \in C$ if $C$ is closed and convex in finite dimensions.

\begin{proposition}
    \label{prop:acc-point-sing-vec}
    Let $(v^k)_{k \in \NN}$ be the sequence of random variables 
    generated by Algorithm~\ref{alg:mafno-orth-exact}.
    Then,
    there exists a subsequence $(v^{k_j})_{j \in \NN}$
    and a real valued random variable $\sigma^2$ with values in the set of singular values of $A$
    such that 
    \[
        \dist(E_{\sigma}, v^{k_j}) \to 0, \quad j \to \infty,
        \quad \textrm{and} \quad 
        \norm{A v^k} \nearrow \sigma,
        \quad k \to \infty,
        \quad \textrm{a.s.}
    \]
\end{proposition}
\begin{proof}
    By Lemma~\ref{lem:a-to-0} we have $a_k^2 \to 0$ for $k \to \infty$ almost surely. 
    As $a_k^2 \leq \|A\|^4$, 
    Lebesgue's dominated convergence theorem shows that $\EE[a_k^2]\to 0$ for $k \to \infty$. 
    Hence, 
    by Lemma~\ref{lem:E_a_k_sqr} 
    and the law of total expectation we obtain
    \begin{align}
        \label{eq:e_a_k_P_v}
        \frac{1}{d-1} \EE\Big[ \|\big(I_d-v^k(v^k)^*\big)A^*A v^k\|^2 \Big] = 
        \EE\Big[\EE\big[ a_k^2 \mid v^k \big]\Big] =
        \EE\big[ a_k^2\big] \to 0,
    \end{align}
    for $k \to \infty$.
    This, in turn, shows the existence
    of a subsequence 
    such that
    \[
        \|\big(I_d-v^{k_j}(v^{k_j})^*\big)A^*A v^{k_j}\| \to 0
        \qquad 
        j \to \infty,
        \quad \textrm{a.s.}
    \]
    Equivalently,
    we have that
    \[
        \textrm{dist}(\ker(I_d - v^{k_j}(v^{k_j})^*), A^*A v^{k_j}) 
        = \textrm{dist}(\Spann(v^{k_j}), A^*A v^{k_j}) 
        \to 0,
        \quad j \to \infty,
        \quad \textrm{a.s.}
    \]
    Hence,
    there exists a sequence $(\sigma_j)_{j \in \NN}$ of random variables
    such that 
    \[
        \norm{A^*A v^{k_j} - \sigma_j^2 v^{k_j}}
        \to 0,
        \qquad j \to \infty,
        \quad \textrm{a.s.}
    \]
    By the almost sure convergence of the sequence $(\norm{A v^k})_{k \in \NN}$,
    cf.~Corollary~\ref{cor:no-sing-vectors-during-iteration},
    we also have the almost sure convergence of the
    $(\norm{A v^{k_j}})_{k \in \NN}$,
    and hence the almost sure convergence of $\sigma_j \to \sigma$ 
    as $j \to \infty$.
    This implies
    \[
        \norm{A^*A v^{k_j} - \sigma^2 v^{k_j}}
        \to 0,
        \qquad j \to \infty,
        \quad \textrm{a.s.}
    \]
    The latter yields that $\sigma$ takes values in the singular values of $A$.
    Hence, we have $\dist(E_\sigma, v^{k_j}) \to 0$ for $j \to \infty$ almost surely,
    which proves the first claim.
    
    The almost sure convergence of the entire sequence $\norm{A v^k}^2 \to \sigma$ as $k \to \infty$
    follows by the monotonicity since for any $k \in \NN$ there exists $j \in \NN$ such that $k_j < k < k_{j+1}$,
    which finishes the proof.
\end{proof}

Since we now know that there exists a subsequence 
converging to a right singular space
and that $\norm{A v^{k}}$ is strictly ascending,
we can conclude that we will reach the global maximum $\norm{A}$, 
once we are past the second largest singular value
and moreover, 
any convergent subsequence of $v^k$ is then converging to the singular space for the largest singular value.

\begin{lemma} 
  \label{lem:B_Lemma}
   Let $(v^k)_{k \in \NN}$ be the sequence of random variables generated by Algorithm~\ref{alg:mafno-orth-exact}
    and 
    \begin{align*} 
        B \coloneqq \{ v \in\sphere^{d-1}: \norm{A v}^2 > \sigma_2^2\}.
    \end{align*}
  Then, the following holds:
  \begin{enumerate}
  \item[(i)] If $v^{k}\in B$, then $v^{k+1}\in B$.
  \item[(ii)] If $v^{k_0}\in B$ holds for some $k_0$, then holds $\dist(E_{\sigma_1}, v^k) \to 0$ for $k \to \infty$ almost surely,
        i.e. $\lim_{k\to\infty} \norm{A v^{k}} = \sigma_1 = \norm{A}$ almost surely.
  \end{enumerate}
\end{lemma}
\begin{proof}\mbox{}
  \begin{itemize}
  \item[(i)] 
    This holds, 
    since $\norm{Av^{k}}^2$ is increasing by construction, see Corollary~\ref{cor:no-sing-vectors-during-iteration}.
  \item[(ii)] 
    By Proposition~\ref{prop:acc-point-sing-vec}
    and utilizing (i),
    we conclude that at least the subsequence $(v^{k_j})_{j \in \NN} \subset (v^k)_{k \in \NN}$ converges almost surely into $E_{\sigma_1}$,
    i.e. it holds that $\dist(E_{\sigma_1}, v^{k_j}) \to 0$ for $j \to \infty$ almost surely.
    
    Now,
    using that almost every $v^k$ can be decomposed 
    into suitable orthonormal right singular vectors of $A$
    with random variables, i.e. coordinate vectors, $\alpha^k \in \sphere^{s-1}$,
    with $v^k = \sum_{i = 1}^s \alpha_i^k v_i^k$ 
    where $s$ is the number of distinct singular values of $A$ and a choice of right singular vectors $v_i^k \in \sphere^{d-1}$.
    To this end, and by assumption, holds
    \[
        \sigma_2^2 
        < \norm{A v^k}^2
        = \norm[\big]{\sum_{i = 1}^s \alpha_i^k A v_i^k}^2
        = \sum_{i = 1}^s (\alpha_i^k)^2 \sigma_i^2
        \leq \sigma_1^2
        \qquad \forall k \geq k_0,
        \quad \textrm{a.s}
    \]
    From the convergence to $\sigma_1$ (see Proposition~\ref{prop:acc-point-sing-vec})
    follows the almost sure convergence of the coordinates of $\alpha^k$ to zero
    which do not correspond to the leading singular value,
    i.e. for all $j \neq 1$ we obtain $\alpha_j^k \to 0$ 
    for $k \to \infty$ almost surely
    and hence $\dist(E_{\sigma_1}, v^k) \to 0$, as $k \to \infty$ a.s.
    \qedhere
  \end{itemize}
\end{proof}

Now comes a technical lemma that will help us to prove that in each step there is a non-zero probability, which is bounded from below by an absolute constant, that the next iterate will fulfill $\norm{Av^{k+1}}^{2}>\sigma_{2}^{2}$.
Roughly speaking, the lemma says that the probability to reach a neighborhood of two antipodal points by a step in a direction taken uniformly at random with appropriate stepsize has an absolute lower bound that does depend neither on the point where we start nor on the point at which we aim.

\begin{figure}
  \centering
    \pgfmathsetmacro{\azimuth}{120}
  \pgfmathsetmacro{\elevation}{30}
  \tdplotsetmaincoords{\azimuth}{\elevation}
  \begin{tikzpicture}[
    scale=1.8,
    tdplot_main_coords
    ,declare function = {
      spherex(\azimuth,\elevation) = cos(\elevation) * cos(\azimuth); 
      spherey(\azimuth,\elevation) = cos(\elevation) * sin(\azimuth); 
      spherez(\azimuth,\elevation) = sin(\elevation); 
    }
    ]
    \draw[->,thick] (0,0,0) -- (-1.5,0,0)node[left]{$x_{2}$};
    \draw[->,thick] (0,0,0) -- (0,1.7,0)node[right]{$x_{3}$};
    \draw[->,thick] (0,0,0) -- (0,0,2)node[right]{$x_{1}$};
    \foreach \longitude in {0,15,...,345} {
      \draw[
      domain = -90:90,
      variable = \latitude,
      very thin, black!30
      ] plot (
      {spherex(\longitude, \latitude)}
      ,{spherey(\longitude, \latitude)}
      ,{spherez(\longitude, \latitude)}
      );
    }
    \foreach \latitude in {-90,-75,...,75} {
      \draw[
      domain = 0:360,
      variable = \longitude,
      very thin, black!30
      ] plot (
      {spherex(\longitude, \latitude)}
      ,{spherey(\longitude, \latitude)}
      ,{spherez(\longitude, \latitude)}
      );
    }
    \draw[
    tdplot_screen_coords
    ] (0,0) circle[radius = 1];
    \newcommand{\thet}{10}
    \pgfmathsetmacro{\azimuthA}{90}
    \pgfmathsetmacro{\elevationA}{0}
    \pgfmathsetmacro{\azimuthB}{180}
    \pgfmathsetmacro{\elevationB}{90-\thet}
    \pgfmathsetmacro{\azimuthC}{0}
    \pgfmathsetmacro{\elevationC}{90-\thet}
    \newcommand{\alpp}{30}
    \pgfmathsetmacro{\azimuthD}{90}
    \pgfmathsetmacro{\elevationD}{\alpp}

    \pgfmathsetmacro{\Ax}{spherex(\azimuthA, \elevationA)}
    \pgfmathsetmacro{\Ay}{spherey(\azimuthA, \elevationA)}
    \pgfmathsetmacro{\Az}{spherez(\azimuthA, \elevationA)}

    \pgfmathsetmacro{\Bx}{spherex(\azimuthB, \elevationB)}
    \pgfmathsetmacro{\By}{spherey(\azimuthB, \elevationB)}
    \pgfmathsetmacro{\Bz}{spherez(\azimuthB, \elevationB)}

    \pgfmathsetmacro{\Cx}{spherex(\azimuthC, \elevationC)}
    \pgfmathsetmacro{\Cy}{spherey(\azimuthC, \elevationC)}
    \pgfmathsetmacro{\Cz}{spherez(\azimuthC, \elevationC)}

    \pgfmathsetmacro{\Dx}{spherex(\azimuthD, \elevationD)}
    \pgfmathsetmacro{\Dy}{spherey(\azimuthD, \elevationD)}
    \pgfmathsetmacro{\Dz}{spherez(\azimuthD, \elevationD)}

    \tdplotsetthetaplanecoords{0}
    \tdplotdrawarc[tdplot_rotated_coords,red]{(0,0,0)}{.35}{0}{-\thet}{left}{$\theta$}
    \fill[
    shift = {(\Ax,\Ay,\Az)},tdplot_screen_coords,blue!80] (0,0) circle(0.7pt)node[right]{$v$};
    \fill[
    shift = {(\Bx,\By,\Bz)},tdplot_screen_coords] (0,0) circle(0.7pt)node[left]{$\tilde{v}^{+}$};
    \draw[thick] (0,0,0) -- (\Bx,\By,\Bz);
    \fill[
    shift = {(\Cx,\Cy,\Cz)},tdplot_screen_coords] (0,0) circle(0.7pt)node[right]{$\tilde{v}^{-}$};
    \draw[thick] (0,0,0) -- (\Cx,\Cy,\Cz);
    \fill[
    shift = {(\Dx,\Dy,\Dz)},tdplot_screen_coords,green!80!black] (0,0) circle(0.7pt)node[right]{$v'$};

    \tdplotcrossprod(\Ax,\Ay,\Az)(\Bx,\By,\Bz)
    \pgfmathsetmacro{\thetheta}{atan2(\tdplotresy, \tdplotresx)}
    \pgfmathsetmacro{\thephi}{atan2(\tdplotresz, sqrt((\tdplotresy)^2 + (\tdplotresx)^2))}
    \tdplotsetrotatedcoords{\thetheta}{90-\thephi}{0}
    \draw[tdplot_rotated_coords,blue!80] (0,0,0) ++ (\Ax,\Ay,\Az) arc (-90:0:-1);

    \tdplotcrossprod(\Ax,\Ay,\Az)(\Cx,\Cy,\Cz)
    \pgfmathsetmacro{\thetheta}{atan2(\tdplotresy, \tdplotresx)}
    \pgfmathsetmacro{\thephi}{atan2(\tdplotresz, sqrt((\tdplotresy)^2 + (\tdplotresx)^2))}
    \tdplotsetrotatedcoords{\thetheta}{90-\thephi}{0}
    \draw[tdplot_rotated_coords,blue!80] (0,0,0) ++ (\Ax,\Ay,\Az) arc (-90:0:-1);

    \draw[gray!90, thin]
        plot[domain=90:0, samples=50]
        ({0},{cos(\x)}, {sin(\x)});

    \pgfmathsetmacro{\dotDB}{\Dx*\Bx + \Dy*\By + \Dz*\Bz}
    \pgfmathsetmacro{\dotDC}{\Dx*\Cx + \Dy*\Cy + \Dz*\Cz}
    \pgfmathsetmacro{\dotDBcl}{max(min(\dotDB,1),-1)}
    \pgfmathsetmacro{\dotDCcl}{max(min(\dotDC,1),-1)}
    \pgfmathsetmacro{\omegaDB}{acos(\dotDBcl)}
    \pgfmathsetmacro{\omegaDC}{acos(\dotDCcl)}
    \pgfmathsetmacro{\sinomegaDB}{sin(\omegaDB)}
    \pgfmathsetmacro{\sinomegaDC}{sin(\omegaDC)}

    \draw[green!80!black,domain=0:1,variable=\t,smooth]
    plot
    (
    { ( ( sin((1-\t)*\omegaDB) / \sinomegaDB ) * \Dx + ( sin(\t*\omegaDB) / \sinomegaDB ) * \Bx ) }
    ,
    { ( ( sin((1-\t)*\omegaDB) / \sinomegaDB ) * \Dy + ( sin(\t*\omegaDB) / \sinomegaDB ) * \By ) }
    ,
    { ( ( sin((1-\t)*\omegaDB) / \sinomegaDB ) * \Dz + ( sin(\t*\omegaDB) / \sinomegaDB ) * \Bz ) }
    );
    
    \draw[green!80!black,domain=0:1,variable=\t,smooth]
    plot
    (
    { ( ( sin((1-\t)*\omegaDC) / \sinomegaDC ) * \Dx + ( sin(\t*\omegaDC) / \sinomegaDC ) * \Cx ) }
    ,
    { ( ( sin((1-\t)*\omegaDC) / \sinomegaDC ) * \Dy + ( sin(\t*\omegaDC) / \sinomegaDC ) * \Cy ) }
    ,
    { ( ( sin((1-\t)*\omegaDC) / \sinomegaDC ) * \Dz + ( sin(\t*\omegaDC) / \sinomegaDC ) * \Cz ) }
    );

    \tdplotsetrotatedcoords{0}{0}{90}
    \fill[tdplot_rotated_coords,thick,red,red,opacity=0.2] (0,0,{cos(\thet)}) circle[radius = {sin(\thet)}];
    \draw[tdplot_rotated_coords,thick,red] (0,0,{cos(\thet)}) circle[radius = {sin(\thet)}]node[above right=2mm]{$B_{\theta}(\tilde{v})$};

    \tdplotsetthetaplanecoords{0}
    \tdplotdrawarc[tdplot_rotated_coords,red]{(0,0,0.95)}{.35}{-\thet}{0}{left}{$\beta$}

  \end{tikzpicture}

    \caption{%
    Visualization in 3d
    of the situation in the proof of Lemma~\ref{lem:inf_probability}.
    }
    \label{fig:sphere_tangent_cone}
  \end{figure}

  \begin{lemma} \label{lem:inf_probability}
  Let $\theta>0$ and $v,\tilde v\in\sphere^{d-1}$ and define the hyperspherical cap around $\tilde{v}$ with polar angle $\theta$ by
    \begin{align*}
      B_{\theta}(\tilde{v}) &\coloneqq \left\{ w\in\sphere^{d-1}\ \middle|\ \scp{w}{\tilde{v}} \geq \cos(\theta) \right\} 
    \end{align*}
    further define
    \begin{align*}
    D_v & := \left\{x\in \{v\}^{\bot} \,\cap\, \sphere^{d-1} \;\middle|\;  \frac{v+\tau x}{\norm{v+\tau x}} \in B_{\theta}(\tilde v)\cup B_{\theta}(-\tilde v) \text{ for some } \tau\in\RR\right\}
  \end{align*}
  and denote by $\mathcal{U}_{v}$ the uniform distribution on the unit sphere in $\{v\}^{\bot}$. 
 Then there exists a constant $p_{\theta,d}$ independent of $v$ and $\tilde{v}$ such that 
  \begin{align}\label{eq:def_p_varepsilon}
    \inf_{v\in\sphere^{d-1}} \mathcal{U}_v\big( D_v \big) = p_{\theta,d}>0.
  \end{align}
\end{lemma}

\begin{proof}
  To estimate the probability, consider the sketch in Figure~\ref{fig:sphere_tangent_cone}.
    The set $D_v$ of directions perpendicular to $v$ that point towards the neighborhood $B_{\theta}(\tilde{v})$ of $\tilde{v}$ is a hyperspherical cap within the unit sphere in $\{v\}^\bot$. 
    This cap is described by a polar angle, we are interested in the one which is described by a lower bound of those. 
    Notably, we can equivalently describe this cap via all great circles through $v$ that intersect $B_{\theta}(\tilde{v})$.
    These great circlesare the edges of the projected cone (onto the $x_{1}$-...-$x_{d-1}$-plane), which define the hyperspherical cap.
    An estimate from below for the polar angle of this hyperspherical cap is given by the angle between the two great circles through $v$ that go through the points $\tilde{v}$ and $\tilde{v}^-$ in the sketch, respectively (this angle is denoted $\beta$ in Figure~\ref{fig:sphere_tangent_cone}).
    As $v$ approaches the equator perpendicular of $\tilde{v}$, the angle $\beta$ decreases (see Figure~\ref{fig:sphere_tangent_cone} and consider the vectors $v$ and $v'$). 
    At the equator, it attains the lower bound of $\theta$ (viewing the situation projected onto the $x_{1}$-...-$x_{d-1}$-plane).

    Consequently, the value $\inf_{v\in\sphere^{d-1}} \mathcal{U}_{v}(D_{v})$ is lower bounded by the weighted area of a hyperspherical cap in a sphere $\sphere^{d-2}$ with angle $\theta$. This weighted area is known to be expressed via the regularized incomplete Beta function (see, e.g.~\cite{frankl1990some})
    \begin{align*}
        p_{\theta,d} \geq \mathrm{I}_{\sin(\theta)^{2}}(\tfrac{d-2}2,\tfrac12)>0.
    \end{align*}
\end{proof}

\begin{remark}
  As shown in Lemma~\ref{lem:lower-bound-reg-inc-Beta} in Appendix \ref{sec:appendix} for $d\geq 3$ it holds that 
  \begin{align*}
   \mathrm{I}_{\sin(\theta)^{2}}(\tfrac{d-2}2,\tfrac12) \geq \frac{1}{2\sqrt{\pi d}}\sin(\theta)^{d-2}\quad \textnormal{and thus,}\quad  p_{\theta,d}\geq \frac{1}{2\sqrt{\pi d}}\sin(\theta)^{d-2}.
  \end{align*}
\end{remark}

Now we are able to prove the main theorem about almost sure convergence of the algorithm,
i.e. $\norm{A v^k} \to \norm{A}$ for $k \to \infty$ a.s.

\begin{theorem}
  \label{thm:convergence_almost_surely}
  Let $(v^k)_{k \in \NN}$ be the sequence of random variables
    generated by Algorithm~\ref{alg:mafno-orth-exact}.
    It holds $\lim_{k\to\infty} \norm{A v^{k}} = \sigma_1$ almost surely.
    Moreover, if the singular space to $\sigma_1$ is one-dimensional, then $(v^{k})_{k \in \NN}$ converges almost surely to a right singular vector of $\sigma_1$. 
\end{theorem}
\begin{proof}
  We consider the set $B$ from Lemma~\ref{lem:B_Lemma} and define the events
  \begin{align*}
    \Omega_k = \{ w \in \Omega : v^{k}(w) \not\in B \}.
  \end{align*}
  Lemma~\ref{lem:B_Lemma}(i) shows $\Omega_{k+1}\subseteq \Omega_k$ and by Lemma~\ref{lem:B_Lemma}(ii) we have
  \begin{align*}
    \big\{ w \in \Omega : \norm{Av^{k}(w)} \not\to \sigma_1 \big\} = \bigcap_{k=0}^\infty \Omega_k.
  \end{align*}
  Now we show that there is some $p$ independent of $k$ and $w \in \Omega$ such that we have 
\begin{align}\label{eq:prob-geometric-decrease}
\PP(\Omega_{k+1}) \leq (1-p)\cdot\PP(\Omega_{k}).
\end{align}
Then the assertion follows, as from $\Omega_{k+1}\subseteq \Omega_k$ we obtain that
  \begin{align*}
    \PP(\norm{A v^k} \not\to \sigma_1)
    = \PP\Big( \bigcap_{k=0}^\infty \Omega_k \Big) = \lim_{k\to\infty} \PP(\Omega_k) = 0.
  \end{align*}

  We prove~\eqref{eq:prob-geometric-decrease} as follows:
    By $U(v,x)$ we denote the update in Algorithm~\ref{alg:mafno-orth-exact} of~$v$ 
    with a direction $x\in \sphere^{d-1}\cap \{v\}^\perp$,
    i.e. $U(v, x) \coloneqq \tfrac{v + \tau x}{\sqrt{1 + \tau^2}}$
    with deterministic step size $\tau$ once $v$ and $x$ are fixed
    given by Proposition~\ref{prop:stepsize-formula}.
    By $\PP_{v}$ we denote the probability functions for $v$,
    and let $K_{x\mid v}$ be the Markov kernel defined by the push forward of the uniform distribution on $\sphere^{d-1}$ through the map $T_{v}y = (y - \scp{v}{y}v)/\norm{y - \scp{v}{y}v}$, i.e. the conditional probability function of $x$ given $v$.

We obtain by Lemma~\ref{lem:B_Lemma} (ii) and the law of total probability 
\begin{align}
        \PP(\Omega_{k+1})
        &= \int\limits_{\sphere^{d-1}} \ \int\limits_{\sphere^{d-1}\cap \{v\}^\perp} 
        \mathbbm{1}_{U(v,x) \not \in B} \  K_{x\mid v}(v, \ \dd x) \ \dd \PP_{v^k}(v) \notag \\
        &= \int\limits_{\sphere^{d-1}\setminus B} \; \int\limits_{\sphere^{d-1}\cap \{v\}^\perp} \mathbbm{1}_{U(v, x) \notin B} \ K_{x\mid v}(v, \ \dd x)\ \dd \PP_{v^{k}}(v).
        \label{eq:prob-Ak+1}
\end{align}
Now we use Lemma~\ref{lem:inf_probability} as follows:
As $B$ is relatively open in $\sphere^{d-1}$ and contains a right singular vector $v_{1}$ for the singular value $\sigma_{1} = \norm{A}$, 
    there exists $\theta>0$ with 
        \begin{align}
    	\label{eqn:relation_B_Bepsilon}
    	C_\theta \coloneqq \Big( B_\theta(v_1) \cup B_\theta(-v_1)\Big)\cap \sphere^{d-1}\subset B. 
    \end{align}
    Lemma~\ref{lem:inf_probability} 
    with this $\theta$ and $v^*=v_1$ gives the lower bound $p_{\theta,d}>0$ on the probability of sampling a direction $x$ for which there is a stepsize that will lead to a new point in $B_{\theta}(v_{1})$. 
    The update of the algorithm maximizes $\norm{Av^{k+1}}$ over the stepsize and effectively maximizes the value of $\norm{Av}$ for $v$ on a semicircle with midpoint $v^{k}$ in direction $x^{k}$.
    Thus, for each direction $x$ for which a step into~$B$ exists,
    the algorithm will choose a stepsize that will land in $B$ as well.
   Hence, we get from the lower bound from Lemma~\ref{lem:inf_probability}, \eqref{eqn:relation_B_Bepsilon} and~\eqref{eq:prob-Ak+1} (using Fubini's lemma) that 
\begin{align*}
\PP(\Omega_{k+1}) 
& \leq \int\limits_{\sphere^{d-1}\setminus B} \; \int\limits_{\sphere^{d-1}\cap \{v\}^\perp} \mathbbm{1}_{U(v, x) \notin C_\theta}\ K_{x\mid v}(v, \ \dd x)\ \dd \PP_{v^{k}}(v) \\
&= \int\limits_{\sphere^{d-1}\setminus B} K_{x \mid v}(v, U^{-1}(v, C_\theta^c)) \ \dd \PP_{v^{k}}(v) \\
&\leq \int\limits_{\sphere^{d-1}\setminus B} (1-p_{\theta,d})\ \dd\PP_{v^{k}}(v) 
= (1-p_{\theta,d})\PP(\Omega_{k}),
\end{align*}
by the law of total probability, again, which proves \eqref{eq:prob-geometric-decrease}.

    Now assume that the singular space to $\sigma_{1}$ is one-dimensional and spanned by $v_{1}\in \sphere^{d-1}$. Then we have that 
\begin{align*}
\big\{ w \in \Omega : \norm{Av^{k}(w)} \not\to \sigma_1 \big\} 
= \{ w \in \Omega : v^{k}(w)\not\to v_{1}\}\cap\{w \in \Omega : v^{k}(w)\not\to -v_{1}\}.
\end{align*}
However, a sequence $(v^{k})$ of realizations cannot accumulate at both $v_{1}$ and $-v_{1}$: By construction, the sequence $\norm{Av^{k}}$ is increasing and, moreover, two consecutive realizations of iterates $v^{k+1}$ and $v^{k}$ always lie on the same hemisphere, i.e. it holds $\sign(\scp{v^{k}}{v_1}) = \sign(\scp{v^{k+1}}{v_1})$. Hence, once we have $v^{k}\in B_{\theta}(v_{1})$, e.g., the next iterates cannot be in $B_{\theta}(-v_{1})$ since, for $\theta$ small enough, this set is too far apart to get there without decreasing the objective value.
\end{proof}

\begin{remark}[Stopping criteria]
    \label{rem:stop_criteria}
    In Proposition~\ref{lem:E_a_k_sqr} we have seen that 
    $\EE[a_k^{2} \mid v^k] = \tfrac1{d-1}\norm{A^{*}Av^{k} - \norm{Av^{k}}^{2} v^{k}}^{2}$, i.e. a small expectation of $a_{k}^{2}$ shows that $v^{k}$ is close to a singular vector.
    Furthermore,
    Lemma~\ref{lem:a=0-at-sing-vecs} tells us 
    that $a_{k}=0$ only happens at right singular vectors.
    This motivates to try to use the expectation of $a_{k}^{2}$ as stopping criterion,
    i.e. stop if  $(d-1)\EE[a_{k}^{2}]\leq \varepsilon^{2}$ as a stopping criterion.
    It turns out to be practical to start approximating the expectation when we observe that $a_{k}^{2}$ is small:
    If $(d-1)a_{k}^{2}\leq\varepsilon^{2}$, 
    resample $x^{k}$ several times and stop if the observed empirical mean obeys the same bound.
\end{remark}

\begin{remark}[Handling the case of singular values of multiplicity $d-1$]
  \label{rem:multiplicity-d-1}
  The assertion of Lemma~\ref{lem:B_Lemma}(ii) also holds in the case of a singular value of multiplicity $d-1$, implying that Algorithm~\ref{alg:mafno-orth-exact} will also converge in this case:
  Under this assumption, there are only two distinct singular values, and we distinguish two cases: If the smaller singular value has multiplicity $d-1$, then a uniform random initialization $v^{0}\in\sphere^{d-1}$ will almost surely give $\norm{Av^{0}}>\sigma_{2}$ as the corresponding eigenspace has measure zero. Hence, we have $a^{0}\neq 0$ and can perform a step. Then either we are done (which is almost surely not the case) or we have $\sigma_{2}<\norm{Av^{k}}<\sigma_{1}$ and we obtain $\norm{Av^{k}}\to\norm{A}$ from Lemma~\ref{lem:B_Lemma}.
  If the larger singular value has multiplicity $d-1$ we argue as follows:
  First note that $\sigma_1 \geq \norm{A v^0} \geq \sigma_2$ for any $v^0 \in \RR^d$.
  Second, if we initialize $v^{0}$ uniformly on $\sphere^{d-1}$, and since the eigenspace corresponding to $\sigma_1$ is one dimensional, having equality in the lower bound means the $v^0$ is a right singular vector corresponding to $\sigma_2$, where the eigenspace is measure zero such that $\norm{Av^{0}}>\sigma_{2}$ almost surely.
  Moreover, by Lemma~\ref{lem:a=0-at-sing-vecs} we obtain that $a_0 \neq 0$ holds (a.s.). By optimality of the stepsize $\tau_0 \in \RR$ we obtain a direct projection onto the space of normalized maximal singular vectors, which is in this case an equator of the $(d-1)$-sphere. We end up with convergence after one iteration, i.e. $\norm{A v^1} = \sigma_1$. 
\end{remark}

\begin{remark}[Handling the case of singular values of multiplicity $d$---detecting orthonormal columns]
    \label{rem:second_assump_gone}
  A curious fact about Algorithm~\ref{alg:mafno-orth-exact} is as follows: 
  If we sample $v^{0}$ uniformly on the unit sphere and observe $a_0 = 0$, then we almost surely have that the operator~$A$ has orthogonal columns and all have the same norm, i.e. that 
    \begin{align}
    \label{eqn:A_with_ortho_col_and_same_norm}
        A^{*}A = cI_{d}
    \end{align}
    holds for some $c>0$. 
    To see this,
    we argue by contradiction:
    If \eqref{eqn:A_with_ortho_col_and_same_norm} does not hold, 
    then $A^*A$ has at least two distinct eigenvalues 
    and therewith at least two distinct proper spaces of eigenvectors.
    Therefore $v^0$ is,
    as uniformly sampled vector from the $(d-1)$-sphere,
    almost surely not a right singular vector of $A$. 
    Hence, 
    we obtain from Lemma~\ref{lem:a=0-at-sing-vecs} 
    that $a_0=0$ happens with probability zero.
    Equivalently,
    we have 
    \begin{align*}
        \PP(a_0 = 0) > 0
        \quad \Rightarrow \quad 
        A^*A = cI_{d},
    \end{align*}
    which yields the assertion. 
\end{remark}

\begin{remark}[Assumption~\ref{ass:ass_sing_rang} can be dropped]
  \label{rem:assumption-can-be-dropped}
  We started the convergence analysis with Assumption~\ref{ass:ass_sing_rang}, i.e. that $A$ has no singular value of multiplicity $d-1$ or $d$. This was used to ensure that if we have $a^{0}\neq 0$, we will have $a^{k}\neq 0$ throughout the iteration.
  However, now we see that this assumption can be dropped:  Remark~\ref{rem:multiplicity-d-1} shows that the method also works if there is a singular value of multiplicity $d-1$ and Remark~\ref{rem:second_assump_gone} shows the same if we have just a single singular value.
\end{remark}

\begin{remark}[Calculating leading right singular vectors]
  One can extend Algorithm~\ref{alg:mafno-orth-exact} to compute a number of leading singular values and corresponding right singular vectors (if they are all distinct). Once $\sigma_{1}$ and $v_{1}$ have been found, we can restrict Algorithm~\ref{alg:mafno-orth-exact} to the space $v_{1}^{\bot}$ using Gram-Schmidt orthogonalization for $x^{k}$ and find $\sigma_{2}$ and $v_{2}$ in this way. Although this procedure gets more computationally heavy for each new singular value, we may still be able to compute the first few leading right singular vectors.
\end{remark}

\subsection{Convergence rates}
\label{sec:conv_rates}

In this subsection 
we derive two convergence results:
one for the sequence $(a_k)_{k \in \NN}$ generated by Algorithm~\ref{alg:mafno-orth-exact} and the second
for the corresponding singular vector and singular value equation,
i.e. 
\[
    \norm{A^*A v^k - \norm{A v^k}^2 v^k}^2, \quad k \in \NN,
    \quad \text{where}\quad 
    v^k\;\text{is from Algorithm~\ref{alg:mafno-orth-exact}}.
\]

The following proposition 
shows the convergence of the $a_k^2 = \scp{A v^k}{A x^k}^2$ 
from Algorithm~\ref{alg:mafno-orth-exact} towards zero 
in a sublinear rate.

\begin{proposition}\label{prop:ak-conv-rate}
  Let $(a_k)_{k \in \NN}$ be the sequence of random variables
  generated by Algorithm~\ref{alg:mafno-orth-exact}.
  It holds almost surely
  \[
    \min_{0\leq k \leq n} a_k^2 \leq \tfrac{2\norm{A}^{4}}{n+1}.
  \]
\end{proposition}
\begin{proof}
  From \eqref{eq:ab_tau} we have  $a_k\tau_k(a_k\tau_k- b_{k}) = a_k^2$ and
  from Lemma~\ref{lem:ascent-formula} we get that $a_{k}\tau_{k} = \norm{Av^{k+1}}^{2} - \norm{Av^{k}}^{2}\geq 0$ and $a_{k}\tau_{k}\leq\norm{A}^{2}$. We also have $\abs{b_{k}}\leq\norm{A}^{2}$, so we get 
\begin{align*}
a_{k}^{2} = a_k\tau_k(a_k\tau_k- b_{k}) \leq 2\norm{A}^{2}a_{k}\tau_{k} = 2\norm{A^{2}}\left(\norm{A v^{k+1}}^2 -\norm{A v^k}^2\right)
\end{align*}
  We sum this equation from $k=0,\dots,n$ and arrive at 
  \begin{align}
    \sum\limits_{k=0}^{n}a_{k}^{2}\leq 2\norm{A}^2 \norm{Av^{n+1}}^{2}\leq 2\norm{A}^4.
    \label{eq:ak_min_O}
  \end{align}
  Finally, estimating $\sum\limits_{k=0}^{n}a_{k}^{2} \geq (n+1)\min_{0 \leq k \leq n} a_{k}^{2}$ proves the claim.
\end{proof}

\begin{remark}[Convergence rate for $a_k \tau_k$]
    \label{rem:pre_ak_min_O}
    Summing up the equation from Lemma~\ref{lem:ascent-formula} 
    for $k = 0,...,n$ we have $(n+1)\min_{0 \leq k \leq n}a_k \tau_k \leq \sum_{0 \leq k \leq n} a_k \tau_k = \norm{A v^{n+1}}^2 - \norm{A v^0}^2 \leq \norm{A}^2$.
    Hence we also have $\min_{0 \leq k \leq n} a_k \tau_k \leq \tfrac{\norm{A}^2}{n+1}$ (a.s.).
\end{remark}

We can turn the above proposition into a convergence rate for the error in
the eigenvector equation for $A^*A$ with approximative eigenvalue $\norm{A v^k}$
in the $k$-th iteration.

\begin{theorem}
    \label{thm:conv_rate_eigen_equation}
    Let $(v^k)_{k \in \NN}$ be the sequence of random variables
    generated by Algorithm~\ref{alg:mafno-orth-exact}.
    It holds almost surely
    \[
        \min_{0 \leq k \leq n} 
        \PP\Bigl(\norm[\big]{A^*A v^k - \norm{A v^k}^2 v^k}^2 \geq \varepsilon\Bigr)
        \leq \tfrac{2(d-1)\norm{A}^{4}}{(n+1) \varepsilon}.
    \]
\end{theorem}
\begin{proof}
  First,
  from Lemma~\ref{lem:E_a_k_sqr},
  we get 
  \begin{equation}
    \label{eq:cond_EE}
    \EE[a_k^2 \mid v^k]
    = \tfrac{1}{d-1}\norm{\norm{Av^{k}}^{2}v^{k} - A^{*}Av^{k}}^{2}.
  \end{equation}
  From~\eqref{eq:ak_min_O} we get 
    \[
        \sum_{k = 0}^n \EE[a_k^2 ] 
        = \EE\Bigr[\sum_{k = 0}^n a_k^2 \Bigl] 
        \leq  2 \norm{A}^4.
    \]
    Using the law of total expectation and \eqref{eq:ak_min_O} 
    we get
    \[
        \min_{0 \leq k \leq n} \EE\Bigl[\EE\bigl[a_k^2 \mid v^k\bigr]\Bigr]
        = \min_{0 \leq k \leq n} \EE\bigl[a_k^2\bigr]
        \leq \tfrac{2 \norm{A}^4}{n+1}.
    \]
    Inserting \eqref{eq:cond_EE} yields 
    \[
        \min_{0 \leq k \leq n} 
        \EE\Bigl[\norm[\big]{\norm{A v^{k}}^{2} v^{k} - A^{*}A v^{k}}^{2}\Bigr]
        \leq \tfrac{2 (d-1) \norm{A}^4}{n+1}.
    \]
    Now, let $k^*(n)$ be the index for which the latter minimum is attained,
    such that for any $\varepsilon > 0$ holds by Markov's inequality
    \begin{align*}
        \PP\Bigl(\norm[\big]{\norm{A v^{k^*(n)}}^{2} v^{k^*(n)} - A^{*}A v^{k^*(n)}}^{2} \geq \varepsilon \Bigr)
        &\leq \frac{\EE\bigl[\norm[\big]{\norm{A v^{k^*(n)}}^{2} v^{k^*(n)} - A^{*}A v^{k^*(n)}}^{2}\bigr]}{\varepsilon} \\
        &\leq \frac{2 (d-1) \norm{A}^4}{(n+1) \varepsilon}.
        \qedhere
    \end{align*}
\end{proof}

By Theorem~\ref{thm:conv_rate_eigen_equation}, we can substitute $\varepsilon = 1/\sqrt{n+1}$,
which implies an error bound of $1/\sqrt{n+1}$ with probability at least $1 - \bigO(1/\sqrt{n})$.
In other words, with probability at least $1 - p$, the error is $\bigO(1/n)$ for any $p \in (0,1)$.

\section{Numerical experiments}
\label{sec:numerical-experiments}

We present numerical experiments to illustrate the practical performance of Algorithm~\ref{alg:mafno-orth-exact}. The code (Jupyter notebooks with Python code) can be found at
\url{https://github.com/dirloren/matrix_free_norms}.

\subsection{Algorithm~\ref{alg:mafno-orth-exact} can be fast when power iteration is slow}
In the first experiment, we illustrate
that Algorithm~\ref{alg:mafno-orth-exact} may be fast in some cases.
We consider the matrix of the form 
\begin{align}
    \label{eq:matrix_fail_imp}
    A 
    \coloneqq 
    \left[\begin{matrix}
        1 & \varepsilon \\
        0 & 1
    \end{matrix}\right]
\end{align}
where $1 \gg \varepsilon > 0$ is small.
In this case,
the singular values are given by the eigenvalues of the matrix 
\begin{align*}
    B 
    \coloneqq 
    A^{*} A 
    = 
    \left[\begin{smallmatrix}
        1 & \varepsilon \\
        \varepsilon & 1 + \varepsilon^2
    \end{smallmatrix}\right]
    \quad \text{namely} \quad
    \lambda_{\max,\min}=
    \tfrac{\varepsilon^2 \pm \varepsilon\sqrt{\varepsilon^2 + 4} + 2}{2},
\end{align*}
such that 
$\norm{A} = \sqrt{\lambda_{\max}} = \sqrt{1 + \tfrac{\varepsilon^2 + \varepsilon\sqrt{\varepsilon^2 + 4} }{2}}$.

\begin{figure}[b]
  \begin{minipage}{0.48\linewidth}
      \includegraphics[width=\textwidth, clip=true, trim=0pt 0pt 50pt 50pt]{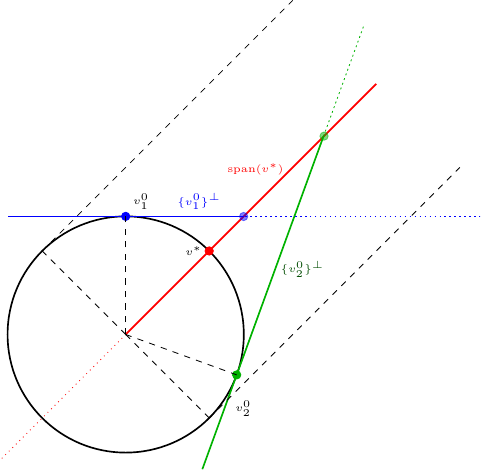}
  \end{minipage}
  \hfill
  \begin{minipage}{0.45\linewidth}
  \caption{Illustration of Algorithm~\ref{alg:mafno-orth-exact} in two dimensions. Almost every initialization (here, two of them are visualized, namely $\{v_1^0, v_2^0\}$, with their corresponding tangent spaces $\{v_1^0\}^\perp$ and $\{v_2^0\}^\perp$) shares a half circle with one of the two global maxima $v^{*}, -v^{*}$. Since the line search is exact, the method will arrive at $v^{*}$ after one step.}
  \label{fig:circle-convergence-algorithm}
  \end{minipage}
\end{figure}

We compare Algorithm~\ref{alg:mafno-orth-exact} and the power iteration.
Algorithm~\ref{alg:mafno-orth-exact} converges in one iteration up to machine precision
(the reason is that the method will explore a half circle by the exact line search and this half circle always contains one of the two solutions $\{v^*,-v^*\}$, cf. Figure~\ref{fig:circle-convergence-algorithm}).
The power iteration will converge to the solution but the convergence rate depends on the ratio of the largest and second largest eigenvalue of $A^{*}A$
in magnitude \cite[Section 7.3]{golub2013}.
In this case, 
we have 
\begin{align*}
  \tfrac{\lambda_{\max}}{\lambda_{\min}} = \tfrac{\varepsilon^2 + \varepsilon\sqrt{\varepsilon^2 + 4} + 2}{\varepsilon^2 - \varepsilon\sqrt{\varepsilon^2 + 4} + 2} 
  = (1 + \tfrac{\varepsilon^{2}}2 + \tfrac{\varepsilon}2\sqrt{\varepsilon^2+4})^{2} = 1 + \bigO(\varepsilon)
\end{align*}
and we see that we will need many iterations for small $\varepsilon$.
In our experiments we reached an error of $10^{-5}$ after about $100$ iterations (depending on the random initialization) for $\varepsilon=10^{-2}$ and already about $1.000$ iterations for $\varepsilon = 10^{-4}$.

\subsection{Simple $0-1$ matrices}

Now we consider the matrix $A = \diag(1,1,0)$ where
the set of maximal right singular vectors is isomorphic to the 1-sphere.
Similarly to the first experiment,
our method converges theoretically in one iteration.
However, if we do not use the stopping criterion from Remark~\ref{rem:stop_criteria}
the iterates will wander around the set of optima, as
can be seen in Figure~\ref{fig:numerical_instabilities}, which results in an exploration 
of the entire unit sphere in the space of right singular vectors.

\begin{figure}
  \begin{minipage}{0.48\linewidth}
  \includegraphics[width=\textwidth]{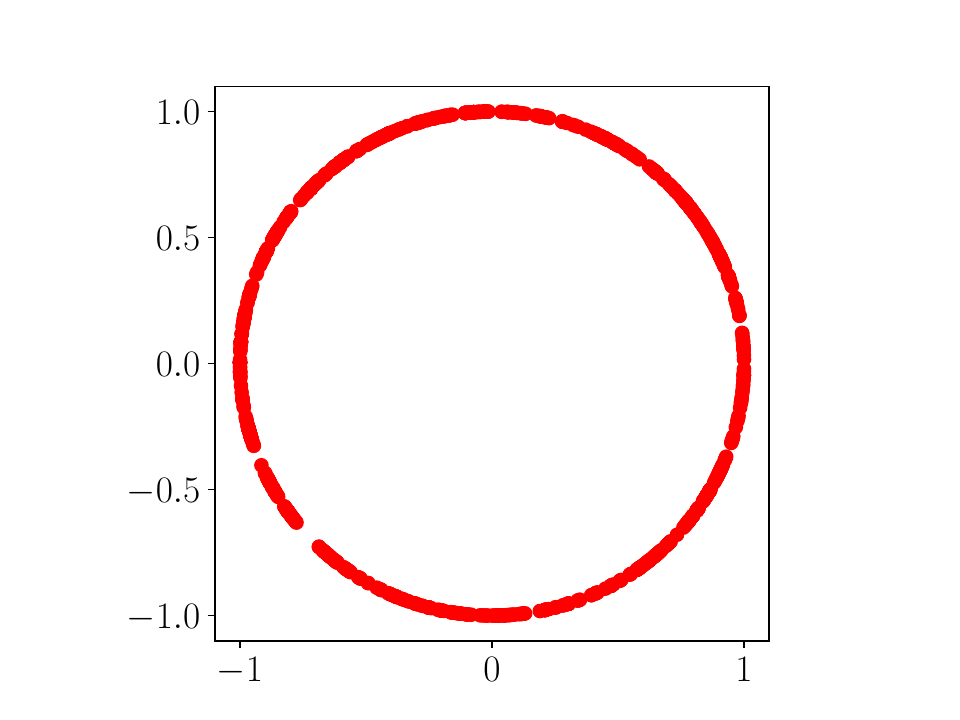}
  \end{minipage}
  \hfill
  \begin{minipage}{0.48\linewidth}
  \caption{Visualization of the great circle in the $x_1$-$x_2$-plane of the $\sphere^2$:
  First two components of the iterates $v^{k}$ (for $k=1,\dots,5000$)
  of Algorithm~\ref{alg:mafno-orth-exact} for $A = \diag(1,1,0)$,
  only plotted when $|v^{k}_3| < \varepsilon$.}
  \label{fig:numerical_instabilities}
  \end{minipage}
\end{figure}

In another experiment,
we consider the matrix $A = [1,0,...,0] \in \RR^{1\times d}$,
which has operator norm $\norm{A} = 1$.
We demonstrate the convergence of our method 
by performing $50$ runs for $d \in \{100, 1000, 10000\}$
of Algorithm~\ref{alg:mafno-orth-exact} with $k = 10d$ iterations.
The resulting (relative) error $(\norm{A}-\norm{A v^k})/\norm{A}$ 
and the mean squared error for the last $d-1$ components
is visualized in Figure~\ref{fig:one_mat}.
We observe linear convergence in all cases and also that the convergence speed is independent of the dimension.
However,
the power iteration on the matrix $A^*A = \diag(1,0,...,0) \in \RR^{d \times d}$
would converge in exactly one iteration,
if the initial vector is uniformly sampled from the sphere $\sphere^{d-1}$.

\begin{figure}[htb]
  \centering
  \begin{tabular}{ccc}
    \qquad $1\times 100$ & \qquad$1\times 1\,000$ & \qquad$1\times 10\,000$\\
  \includegraphics[width=0.3\textwidth]{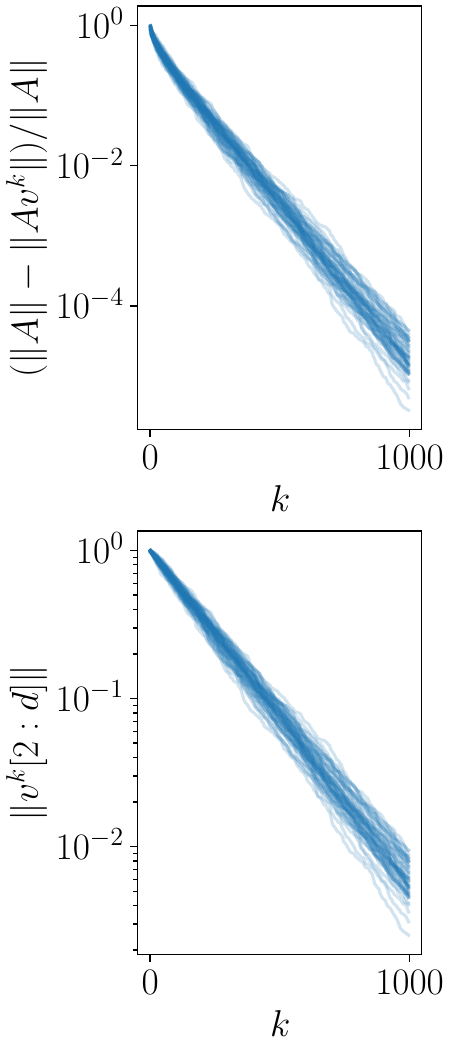}&
  \includegraphics[width=0.3\textwidth]{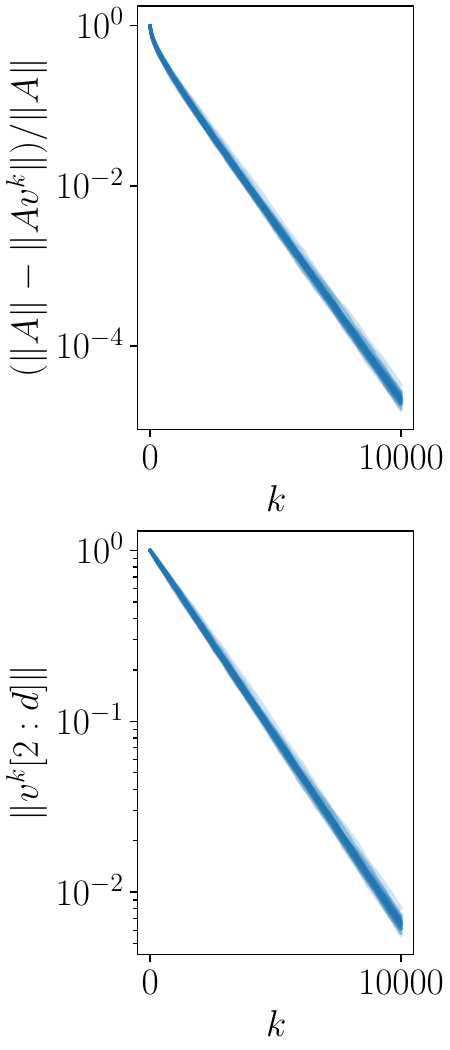}&
  \includegraphics[width=0.3\textwidth]{simax-figures/experiment2_1000.pdf}
  \end{tabular}
  \caption{Results for 50 runs of Algorithm~\ref{alg:mafno-orth-exact} for matrices $A = [1,0,...,0] \in \RR^{1\times d}$. All plots are linear in the $x$-axis and logarithmic in the $y$-axis indicating linear convergence of the quantities in this case.}
  \label{fig:one_mat}
\end{figure}

\subsection{Random Gaussian matrices}

We investigate the convergence rate in dependence on the dimension 
of the input and output space of the operator $A$, respectively.
To that end we generated random matrices of size $m\times d$ for different values of $m$ and $d$, performed $50$ runs of Algorithm~\ref{alg:mafno-orth-exact} with $k=20d$ iterations for each matrix and plot the relative error $(\norm{A}-\norm{Av^{k}})/\norm{A}$, the running minimum over the $a_{k}^{2}$ over iterations in Figure~\ref{fig:experiment3} and the upper bound from Proposition~\ref{prop:ak-conv-rate}.
We observe that the number $m$ of rows does not seem to have a large influence on the convergence speed, but for larger $d$ the method needs more iterations to reach a similar accuracy (note the different scaling of the y-axis).

\begin{figure}[htb]
  \centering
  \begin{tabular}{cc}
    $10\times50$ & $100\times500$\\
    \includegraphics[width=0.47\textwidth]{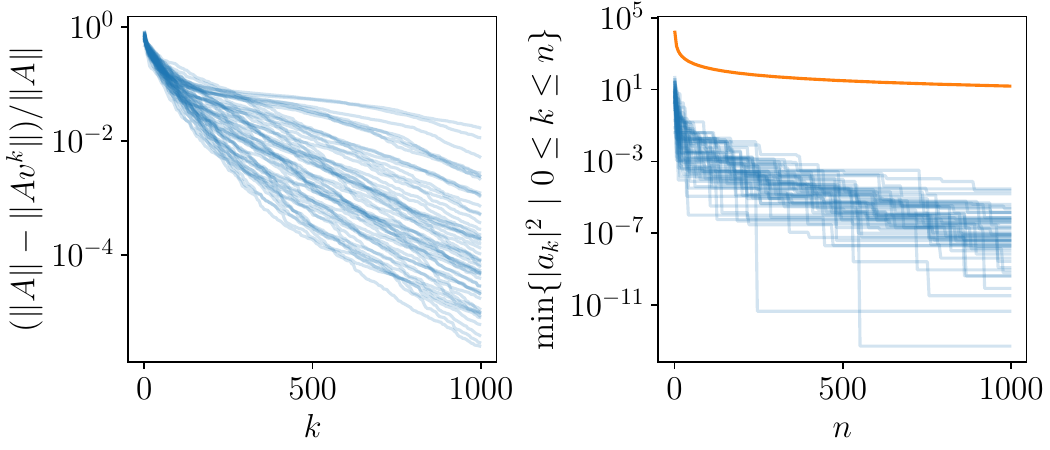}&
    \includegraphics[width=0.47\textwidth]{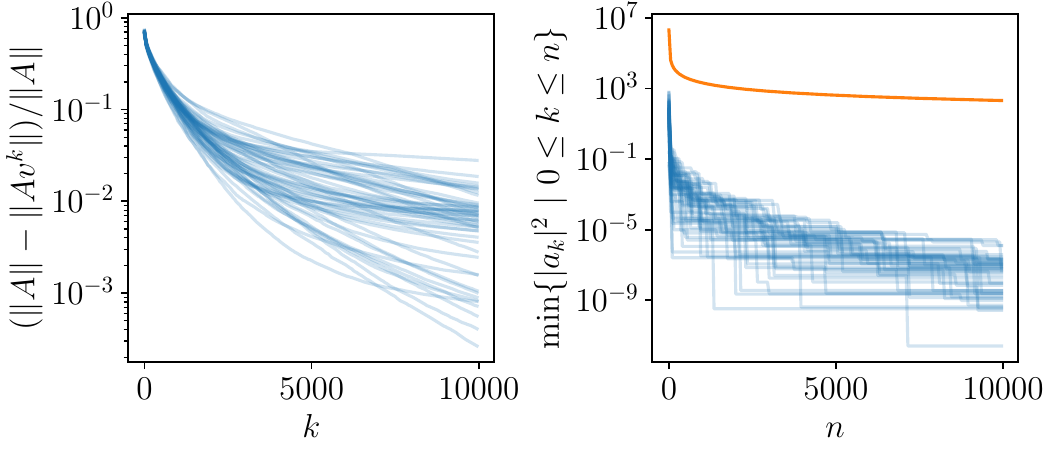}\\
    $50\times50$ & $500\times500$\\
    \includegraphics[width=0.47\textwidth]{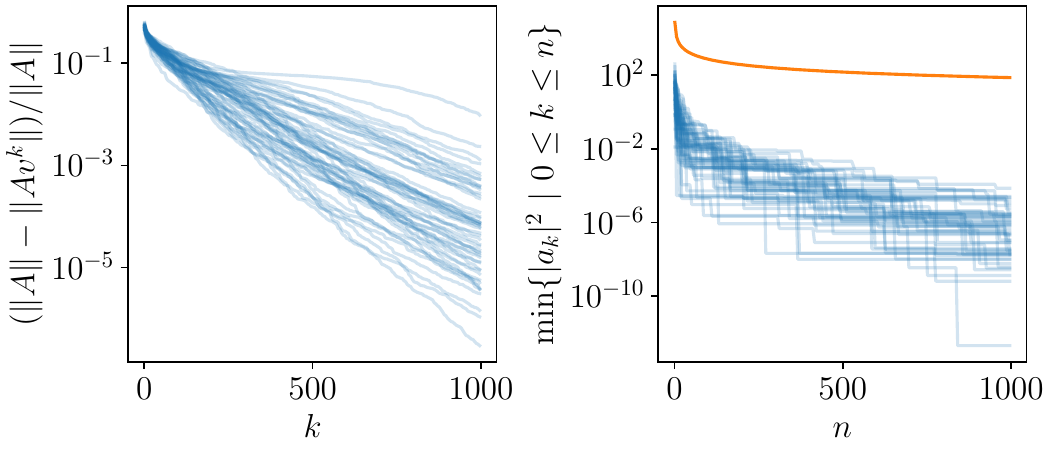}&
    \includegraphics[width=0.47\textwidth]{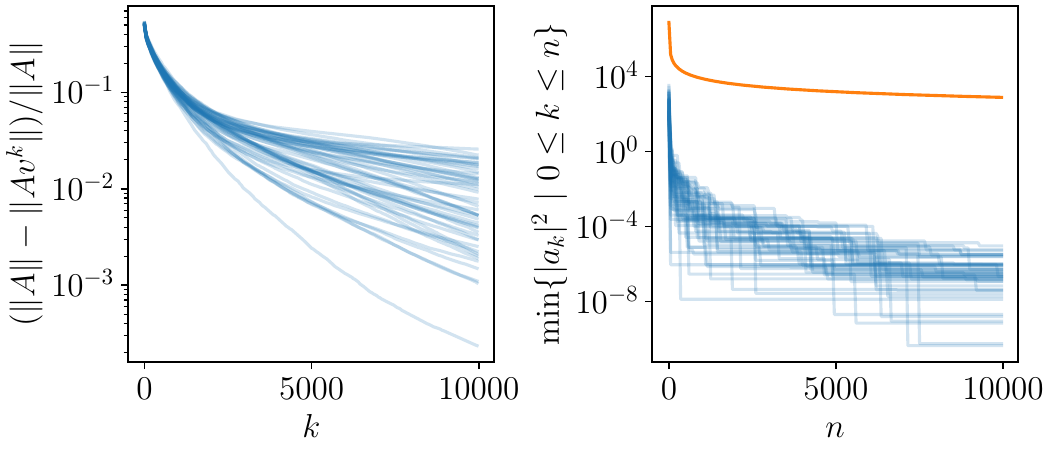}\\
    $100\times50$ & $1000\times500$\\
    \includegraphics[width=0.47\textwidth]{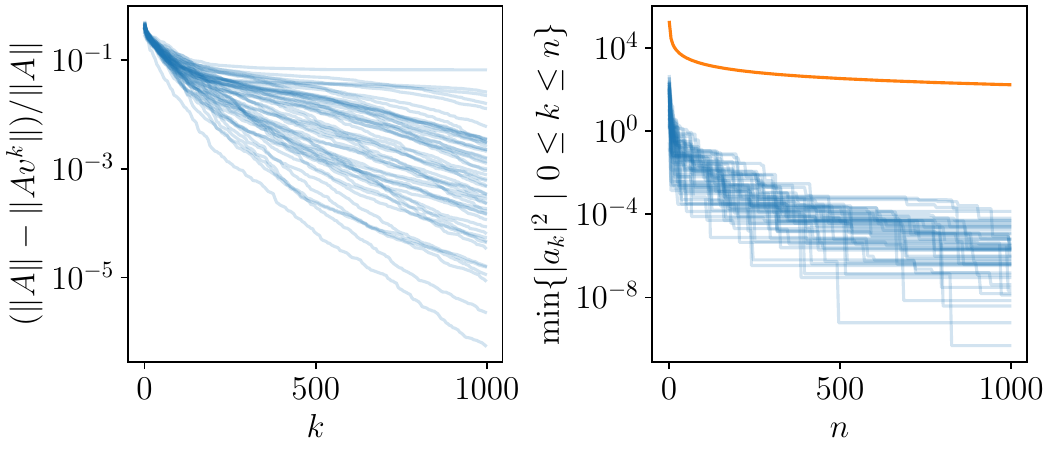} &
    \includegraphics[width=0.47\textwidth]{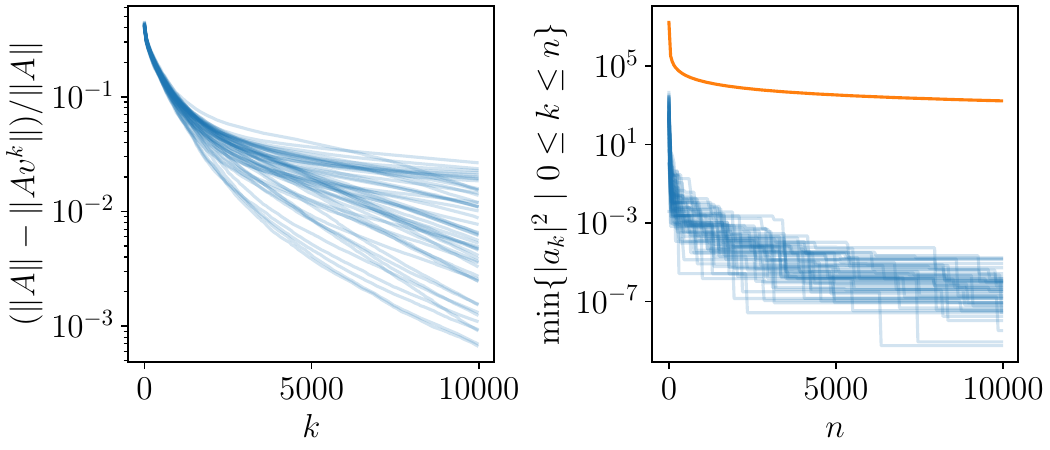}
  \end{tabular}
  \caption{Results for 50 runs of Algorithm~\ref{alg:mafno-orth-exact} for Gaussian matrices of different sizes. Left plots: Relative error over iterations. Right plots: Running minimal value of $\abs{a_{k}}^{2}$ (blue) and upper bound from Proposition~\ref{prop:ak-conv-rate} (orange) over iterations.}
  \label{fig:experiment3}
\end{figure}

\subsection{Implementations of the Radon transform}
\label{sec:implementation-radon}

We come back to the example of the Radon transform in \texttt{skimage} from the introduction and want to clarify what the computed norms in Table~\ref{tab:table-norm-radon} mean. To that end we applied our method to compute the norm of the operators
\begin{lstlisting}[language=Python]
N = 125
numAng = 6
theta = np.linspace(0.0, 180, numAng, endpoint=False)
def A(v):
    return radon(v.reshape(N,N),theta=theta,preserve_range=True).ravel()
def AT(w):
    return iradon(w.reshape(N,numAng),theta=theta,filter_name=None,preserve_range=True).ravel()
\end{lstlisting}
i.e. our dimensions for $A$ and $d = 15625$ and $m = 750$.
(We use a smaller number of angles to speed up computation.)
We run our method using the stopping criterion from Remark~\ref{rem:stop_criteria} with $\varepsilon^{2} = 10^{-8}$ and 10 resamples of $x$.
Our method terminates after $96\,547$ iterations for \texttt{A} and after $5\,599$ iterations for \texttt{AT}. The power iteration for \texttt{A} stabilizes after $10$ iterations and we report the values in Table~\ref{tab:table-norm-radon-2}.

\begin{table}[htb]
  \begin{minipage}{0.48\linewidth}
  \begin{tabular}{lr} \toprule 
    $\norm{A}$ from~\eqref{eq:normA1} & 12.9659\\
    $\norm{A}$ from~\eqref{eq:normA2} & 6.6342\\
    $\norm{A}$ from~\eqref{eq:normA3} & 12.9659\\
    $\norm{A}$ from~\eqref{eq:normA4} & 25.3404\\
    $\norm{A}$ by our method & 25.4766\\
    $\norm{A^*}$ by our method & 6.6342 \\ \bottomrule
  \end{tabular}
  \end{minipage}
  \hfill
  \begin{minipage}{0.51\linewidth}
  \caption{Numerical approximation of the operator norm of the Radon transform in \texttt{scikit-image}  for images of size $125\times 125$ and six equidistant angles by the power method and different choices of the approximation and norms of $A$ and $A^{*}$ by our method.}
  \label{tab:table-norm-radon-2}
  \end{minipage}
\end{table}

It turns out that the power iteration with~\eqref{eq:normA2} accurately computes the norm of the backprojection $A^{*}$ and that~\eqref{eq:normA4} is a fairly accurate lower bound for the norm of $A$. We emphasize that this is not something that we could have concluded without using our method.

We can also compare the different singular vectors, i.e. the limits $v$ of the power iteration and our method. We show both singular vectors and their absolute difference in Figure~\ref{fig:radon-sing-vec}. The difference is a combination of a slightly different structure inside the reconstruction circle and a fairly large difference on the boundary of that circle.
Even though the difference is small in the interior of the circle, the structural difference becomes completely smoothed out by using the approximation of the Radon transformation in the power method. Again, this property would not be recognized without our method.

\begin{figure}[htb]
  \centering
  \includegraphics[width=0.32\textwidth]{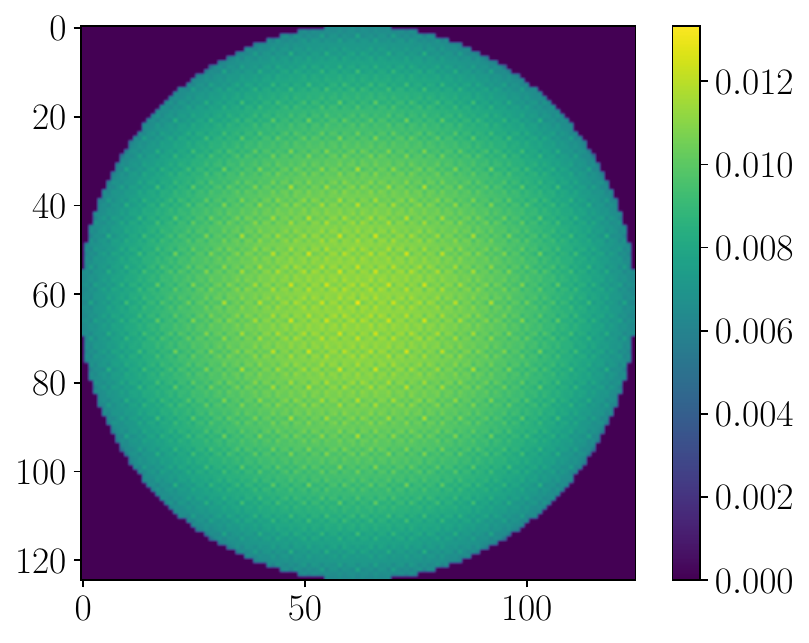}
  \includegraphics[width=0.32\textwidth]{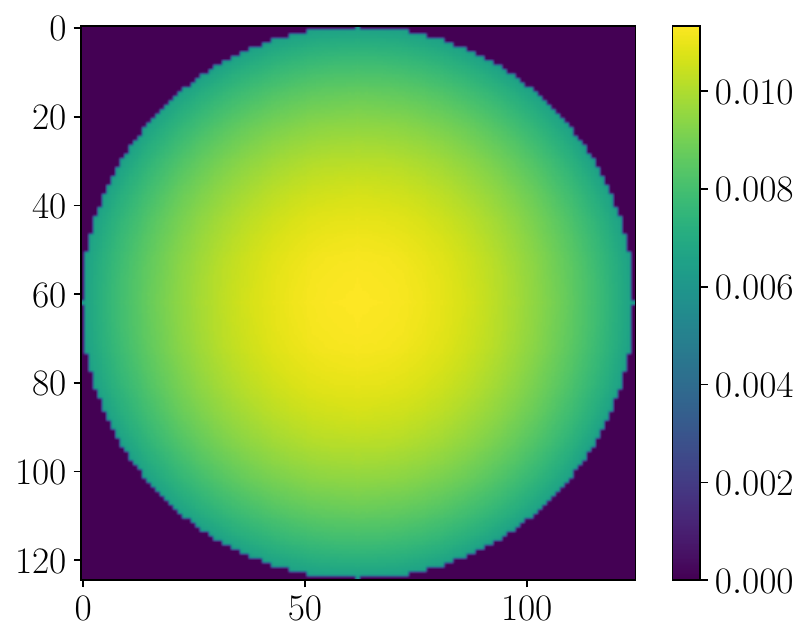}
  \includegraphics[width=0.32\textwidth]{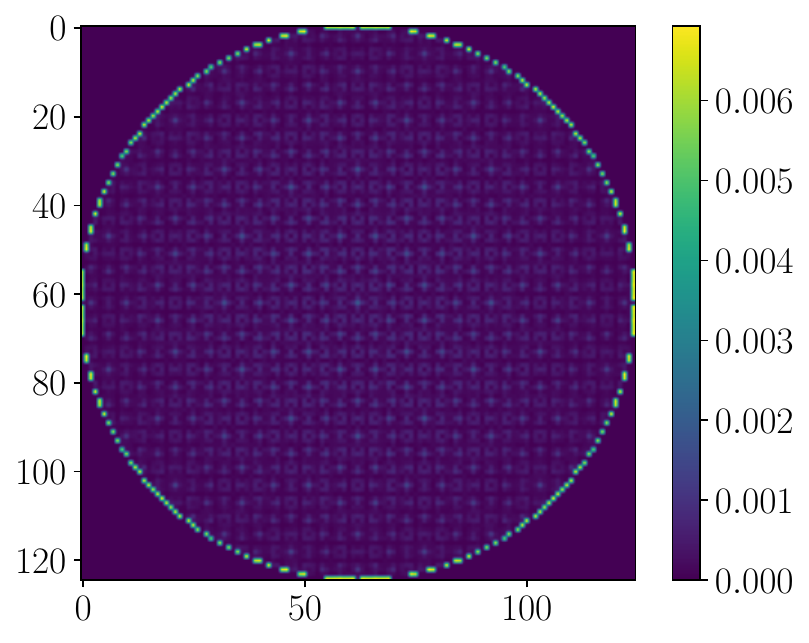}
  \caption{Singular vectors of the Radon transform of \texttt{skimage} for images of size $125\times 125$ and six equidistant angles. Left: Computed by our method. Middle: Computed by power iterations. Right: Difference of both.}
  \label{fig:radon-sing-vec}
\end{figure}

\section{Conclusion and outlook}
\label{sec:conclusion}

Algorithm~\ref{alg:mafno-orth-exact} is a stochastic method for computing the operator norm of a linear map $A$ between real Hilbert spaces and it only uses applications of the operator and computations of norms and inner products of vectors.
We show almost sure convergence of the value $\norm{Av^{k}}$ to the operator norm.

Our examples show that the method works in practice but also show that it may be slow (while it is indeed faster than the power iteration in some special cases).
The question of the convergence speed regarding $\norm{A}-\norm{Av^k}$ is still open;
however, we answered the question of the convergence speed for the corresponding eigenvector and eigenvalue equation.

We only treated real Hilbert spaces here but the case of complex spaces can be treated similarly and only the definition of $a_{k}$ should be changed to $\operatorname{Re}(\scp{Av^{k}}{Ax^{k}})$.

Another question that we did not answer here is the following. How can we calculate $\norm{A-V}$ if we only can do evaluations of $A$ and $V^{*}$ (and no evaluations of $A^{*}$ and $V$ are possible)? This question is important in the case of adjoint mismatch as it happens, for example, in computerized tomography~\cite{lorenz2023chambolle,lorenz2018randomized,chouzenoux2023convergence,Savanier2021ProximalGA,Chouzenout2021pgm-adjoint,zeng2000unmatched}. The idea we presented here can be generalized to this case if we start from the formulation 
\begin{align*}
\norm{A-V} = \max_{\norm{u}=1,\norm{v}=1} \scp{u}{(A-V)v} = \max_{\norm{u}=1,\norm{v}=1}\left( \scp{u}{Av} - \scp{V^{*}u}{v}\right)
\end{align*}
which we worked out in~\cite{bresch2025computing}.

Finally, it is worth mentioning that Algorithm~\ref{alg:mafno-orth-exact} can be adapted to calculate the \emph{smallest} singular value of $A$.
To that end, one only has to replace the $\argmax$ in the calculation of the stepsize with an $\argmin$ which can also be computed exactly. The practicality of this method shall be investigated in upcoming work. 

\paragraph{Acknowledgment} J.~Bresch gratefully thanks Martin Schoen for the correct version of Prop.~\ref{prop:acc-point-sing-vec} and Oleh Melnyk, and Gabriele Steidl for critical reading of this paper.

\appendix

\section{Appendix}
\label{sec:appendix}
Recall that the regularized incomplete Beta function is 
\begin{align*}
\mathrm{I}_x(a,b) \coloneqq \frac{\mathrm{B}(x;a,b)}{\mathrm{B}(a,b)},
\quad a,b > 0, \quad x \in [0,1],
\end{align*}
with the incomplete Beta function \cite[Eq.~(2), III.~§~4]{renyi1966wahrscheinlichkeitsrechnung}
\begin{align*}
\mathrm{B}(x;a,b) \coloneqq \int\limits_0^x t^{a-1}(1-t)^{b-1} dt 
\quad\text{and}\quad 
\mathrm{B}(a,b) \coloneqq \mathrm{B}(1;a,b).
\end{align*}

\begin{lemma}
  \label{lem:lower-bound-reg-inc-Beta}
  For $0\leq x\leq 1$ and $d\geq 3$ it holds that 
  \begin{align*}
  \mathrm{I}_{x}(\tfrac{d-2}2,\tfrac12) \geq \frac{1}{2 \sqrt{\pi d}}\sqrt{x}^{d-2}.
  \end{align*}
\end{lemma}
\begin{proof}
  We first estimate (using that $(1-t)^{-1/2}\geq 1$ for $0\leq t\leq 1$)
  \begin{align}\label{eq:est-inc-beta}
    \mathrm{B}(x;a,\tfrac12) \geq \int\limits_0^x t^{a-1}\dd t = \frac{x^{a}}{a}.
  \end{align}
  Furthermore it holds that \cite[Eq.~(3), III.~§~4]{renyi1966wahrscheinlichkeitsrechnung}
  \begin{align*}
  \mathrm{B}(\tfrac{d-2}2,\tfrac12) = \frac{\Gamma \left( \tfrac{d-2}2 \right) \Gamma \left( \tfrac12 \right)}{\Gamma \left( \tfrac{d-1}2 \right)}.
  \end{align*}
  We use $\Gamma(1/2) = \sqrt{\pi}$ and the Gautschi's inequality \cite[Thm.~3]{elezovic2000gautschi}
  \begin{align*}
  x^{1-s} < (x+\tfrac{s}{2})^{1-s} < \frac{\Gamma(x+1)}{\Gamma(x+s)} < (x+1)^{1-s}
  \end{align*}
  (valid for $0<s<1$) with $x+1 = (d-1)/2$ and $s=1/2$ to estimate 
  \begin{align*}
    \frac{1}{\mathrm{B}(\tfrac{d-2}2,\tfrac12)} 
    & = \frac{\Gamma \left( \tfrac{d-1}2 \right)}{\Gamma \left( \tfrac{d-2}2 \right) \Gamma \left( \tfrac12 \right)}
    \geq \frac1{\sqrt{\pi}} \left( \frac{d-3}{2} + \frac{1}{4}\right)^{1/2}.
  \end{align*}
  Combining this with~\eqref{eq:est-inc-beta} with $a = (d-2)/2$ yields 
  \begin{align*}
  \mathrm{I}_{x}(\tfrac{d-2}2,\tfrac12) \geq \frac{\sqrt{x}^{d-2}}{d-2} \frac2{\sqrt{\pi}} \left( \frac{d-3}{2} + \frac{1}{4}\right)^{1/2}
  \end{align*}
  For $d\geq 3 > \tfrac{20}{7}$ it holds that
  \begin{align*}
    \frac{1}{d-2} > \frac{1}{d}
    \quad\text{and}\quad
    \left(\frac{d-3}{2} + \frac{1}{4}\right)^{1/2} > \frac{\sqrt{d}}{4}
  \end{align*}
  gives the claimed bound.  
\end{proof}

\bibliographystyle{plainurl}
\bibliography{./literature}

\end{document}